\documentclass[reqno,a4paper]{amsart}

\usepackage[latin1]{inputenc}
\usepackage[english]{babel}
\usepackage{eucal,amsfonts,amssymb,amsmath,amsthm,epsfig,mathrsfs}
\usepackage{cancel,soul}
\usepackage{color}
\usepackage{amscd,amsxtra}
\usepackage{enumerate}
\usepackage{latexsym}
\usepackage{bm}

\allowdisplaybreaks

\newcounter{ipotesi}

\Alph{ipotesi}
 \makeatletter \@addtoreset{equation}{section}

\makeatother \makeatletter

\newtheorem{thm}{Theorem}[section]
\newtheorem{hyp}[thm]{Hypotheses}{\rm}
{\rm}
\newtheorem{lemm}[thm]{Lemma}
\newtheorem{cor}[thm]{Corollary}

\newtheorem{prop}[thm]{Proposition}
\newtheorem{defi}[thm]{Definition}
\newtheorem{rmk}[thm]{Remark}{\rm}

\newcounter{parentenv}

\newcommand{\R}{{\mathbb R}}
\newcommand{\CC}{{\mathbb C}}
\newcommand{\E}{{\mathbb E}}
\newcommand{\N}{{\mathbb N}}

\newcommand{\X}{{\mathcal{X}}}
\newcommand{\K}{{\mathcal{K}}}
\newcommand{\D}{{\nabla}}
\newcommand{\J}{{\mathcal{D}}}
\newcommand{\g}{{\mathcal{O}}}

\newcommand{\eps}{\varepsilon}
\newcommand{\ra}{\rightarrow}

\newcommand{\Tr}{{\operatorname{Tr}}}
\newcommand{\tr}{{\operatorname{Tr}}}
\newcommand{\Dom}{{\operatorname{Dom}}}

\newcommand{\Span}{{\operatorname{span}}}
\newcommand{\Id}{{\operatorname{I}}}

\newcommand{\set}[1]{{\left\{#1\right\}}}
\newcommand{\pa}[1]{{\left(#1\right)}}

\newcommand{\gen}[1]{{\left\langle #1\right\rangle}}
\newcommand{\abs}[1]{{\left|#1\right|}}
\newcommand{\norm}[1]{{\left\|#1\right\|}}
\newcommand{\scal}[2]{{\left\langle #1,#2\right\rangle}}
\newcommand{\dscal}[3]{{_{#3}\left\langle #1,#2\right\rangle}_{{#3}^*}}

\newcommand{\mi}[1]{{\lbrace #1(t,x)\rbrace_{t\geq 0}}}
\newcommand{\mii}[1]{{\lbrace #1(t,x)\rbrace_{t\geq 0}}}

\newcommand{\eqsys}[1]{{\left\{\begin{array}{ll}#1\end{array}\right.}}
\newcommand{\tc}{\, \middle |\,}
\newcommand{\qc}{\mathbb{P}\mbox{-a.s.}}
\begin{document}

\frenchspacing

\title[$L^2$-theory for transition semigroups associated to dissipative systems]{$L^2$-theory for transition semigroups associated to dissipative systems}

\author[D. A. Bignamini]{{D. A. Bignamini}}

\address[D. A. Bignamini]{Dipartimento di Scienze Matematiche, Fisiche e Informatiche, Universit\`a degli Studi di Parma, Parco Area delle Scienze 53/A, 43124 Parma, Italy.}
\email{\textcolor[rgb]{0.00,0.00,0.84}{davideaugusto.bignamini@unimore.it}}



\date{\today}

\date{\today}

\begin{abstract}
Let $\X$ be a real separable Hilbert space. Let $C$ be a linear, bounded, non-negative operator on $\X$ and let $A$ be the infinitesimal generator of a strongly continuous semigroup in $\X$. Let $\{W(t)\}_{t\geq 0}$ be a $\X$-valued cylindrical Wiener process on a filtered (normal) probability space $(\Omega,\mathcal{F},\{\mathcal{F}_t\}_{t\geq 0},\mathbb{P})$. Let $F:\Dom(F)\subseteq\X\ra\X$ be a smooth enough function.  We are interested in the generalized mild solution $\mi{X}$ of the semilinear stochastic partial differential equation
\begin{gather*}
\eqsys{
dX(t,x)=\big(AX(t,x)+F(X(t,x))\big)dt+ \sqrt{C}dW(t), & t>0;\\
X(0,x)=x\in \X.
}
\end{gather*}
We consider the transition semigroup defined by
\begin{align*}
P(t)\varphi(x):=\E[\varphi(X(t,x))], \qquad \varphi\in B_b(\X),\ t\geq 0,\ x\in \X.
\end{align*}
If $\g$ is an open set of $\X$, we consider the Dirichlet semigroup defined by
\begin{equation*}
P^{\mathcal{O}}(t)\varphi(x):=\mathbb{E}\left[\varphi(X(t,x))\mathbb{I}_{\{\omega\in\Omega\; :\;\tau_x(\omega)> t\}}\right],\quad \varphi\in B_b(\mathcal{O}),\; x\in\g,\; t>0
\end{equation*}
where $\tau_x$ is the stopping time defined by
\begin{equation*}
\tau_x=\inf\{ s> 0\; : \; X(s,x)\in \g^c \}.
\end{equation*} 
We study the infinitesimal generator of $P(t)$, $P^\g(t)$ in $L^2(\X,\nu)$, $L^2(\g,\nu)$ respectively, where $\nu$ is the unique invariant measure of $P(t)$.
\end{abstract}

\maketitle

\section{Introduction}
Let $\X$ be a real separable Hilbert space with inner product $\scal{\cdot}{\cdot}$ and norm $\norm{\cdot}$. Let \\$(\Omega,\mathcal{F},\{\mathcal{F}_t\}_{t\geq 0},\mathbb{P})$ be a filtered (normal) probability space and let $\{W(t)\}_{t\geq 0}$ be a $\X$-valued cylindrical Wiener process on $(\Omega,\mathcal{F},\{\mathcal{F}_t\}_{t\geq 0},\mathbb{P})$. Let $A:\Dom(A)\subseteq\X\ra\X$ be the infinitesimal generator of strongly continuous semigroup $e^{tA}$ and let $C\in\mathcal{L}(\X)$ (the space of bounded and linear operators from $\X$ to $\X$) be a non-negative operator (so $C$ is self-adjoint). Let $F:\Dom(F)\subseteq\X\ra\X$ (possibly non linear). We introduce the SPDE (Stochastic Partial Differential Equation)
\begin{gather}\label{eqFO}
\eqsys{
dX(t,x)=\big(AX(t,x)+F(X(t,x))\big)dt+\sqrt{C}dW(t), & t>0;\\
X(0,x)=x\in \X.
}
\end{gather}
This type of SPDE is widely studied in the literature, see for example \cite{AD-BA-MA1,BF20,BF2,CER1,DA4,DA1,DA3,DA-DE-GO1,DA-RO1,DA-ZA4,DA-ZA2,ES-ST2,GO-KO1,MAS1,PES-ZA1}. In this paper we focus on the case of dissipative systems, where $A$ and $F$ satisfy a joint dissipativity condition (see Hypotheses \ref{EU2}\eqref{EU2.4}).

If $\Dom(F)=\X$, for any $x\in\X$ it is possible to consider the solution of the mild form of \eqref{eqFO}, namely
\begin{align}\label{mildform}
X(t,x)=e^{tA}x+\int_0^te^{(t-s)A}F(X(s,x))ds+W_A(t),\quad \qc
\end{align}
where $\{W_A(t)\}_{t>0}$ is the process defined by
\begin{align*}
W_A(t):=\int^t_0e^{(t-s)A}\sqrt{C}dW(s).
\end{align*}
However, if $\Dom(F)\subset\X$ is a proper subset of $\X$, \eqref{mildform} may not make sense for every $x\in\X$, since it is not guaranteed that there exists a process $\mi{X}$ such that verifies the \eqref{mildform} and its trajectories live in $\Dom(F)$. So we need a more general notion of solution. Around the nineties S. Cerrai G. Da Prato and J. Zabczyk have considered the notion of generalized mild solution to avoid the problem of $\Dom(F)$. The idea to construct a generalized mild solution is to assume that there exists a Banach space $E\subseteq\Dom(F)$ densely and continuously embedded in $\X$ such that the operator $A$ and the function $F$ have some ``good'' properties $E$. We prove that for any $x\in E$, the SPDE \eqref{eqFO} has a unique mild solution $\mi{X}$ (i.e. check \eqref{mildform}) such that its trajectories take values in $E$. After, exploiting the density of $E$, we will prove that for any $x\in \X$ there exists a process $\mi{X}$, such that
\begin{equation}\label{mildgen}
\lim_{n\ra\infty}\norm{X(\cdot,x_n)-X(\cdot,x)}_{C([0,T],\X)}=0,\quad\forall T>0,\;\qc
\end{equation}
where $\{x_n\}_{n\in\N} \subseteq E$ is a sequence converging to $x$ and $X(t,x_n)$ is the unique mild solution of \eqref{eqFO}, with initial datum $x_n$. We call the limit $\mi{X}$ of \eqref{mildgen} generalized mild solution of \eqref{eqFO}. We refer to \cite[Chapter 7]{CER1} and \cite[Chapter 4]{DA1} for two examples where the generalized mild solution is constructed on an explicit space $E$. 
In Section \ref{GMS}, under suitable hypotheses (Hypotheses \ref{EU2}) we show that, for any $x\in\X$, the SPDE \eqref{eqFO} has a unique generalized mild solution $\mi{X}$.

Let $B_b(\X)$ be the space of bounded and Borel measurable functions from $\X$ to $\R$. We consider the semigroup
\begin{equation}\label{semiX}
P(t)\varphi(x):=\E[\varphi(X(t,x))],\quad \varphi\in B_b(\X),\; x\in\X,\; t>0
\end{equation}

In subsection \ref{misinvS}, with an additional hypothesis (Hypotheses \ref{EU3}), we will prove that the semigroup $P(t)$ has a unique invariant measure $\nu$, such that $\nu(E)=1$, and $\nu$ has finite moments of every order. By the invariance of $\nu$ and standard arguments, $P(t)$ is uniquely extendable to a strongly continuous semigroup $P_p(t)$ in $L^p(\X,\nu)$, for $p\geq 1$. We denote by $N_2$ the infinitesimal generator of $P_2(t)$. A fundamental tool to study the behavior of the transition semigroup in $L^2(\X,\nu)$ will be to define a core of regular functions over which is known the action of $N_2$. The candidate to be the core will be the space 
\[
\xi_A(\X):=\Span\{\mbox{real and imaginary parts of the functions } x\mapsto e^{i\scal{x}{h}}\,|\, h\in \Dom(A^*)\}.
\]
We will prove that, on $\xi_A(\X)$, $N_2$ acts as the following second order Kolmogorov operator defined by
\begin{equation}\label{OPFO}
 N_{0}\varphi(x):=L_0\varphi(x)+\scal{F_0(x)}{\D\varphi(x)},\quad \varphi\in \xi_A(\X),\;x\in\X,
\end{equation}
where
\begin{equation}\label{OPFOC}
L_0\varphi(x):=\frac{1}{2}\tr[C\D^2\varphi(x)]+\scal{x}{A^*\D\varphi(x)}
\end{equation}
Precisely we will prove an even more significant result. 

\begin{thm}\label{identif}
Assume that Hypotheses \ref{EU3var} hold true. $N_2$ is the closure in $L^2(\X,\nu)$ of the operator $N_0$, defined in \eqref{OPFO}. In particular $\xi_A(\X)$ is a core for $N_2$. 
\end{thm}
To prove Theorem \ref{identif}, we should first extend the operator $N_0$ into $L^2(\X,\nu)$, to do so is fundamental that $\nu(E)=1$, since $F$ is well defined only on $E$.
Theorem \ref{identif} extends those contained in \cite[Section 3]{BF2}, \cite{DA3}, \cite[Sections 3.5 and 4.6]{DA1} and \cite[Section 11.2.2]{DA-ZA1}. For a study of an analogous problem in $L^2(E,\nu)$ in the case of a multiplicative noise we refer to \cite{CE-DA1}.

Let $\g$ be an open set of $\X$ and let  $B_b(\g)$ be the space of bounded and Borel measurable functions from  $\g$ to $\R$. In Section \ref{gradsist} we consider the Dirichlet semigroup
\begin{equation}\label{semi2}
P^{\mathcal{O}}(t)\varphi(x):=\mathbb{E}\left[\varphi(X(t,x))\mathbb{I}_{\{\omega\in\Omega\; :\;\tau_x(\omega)> t\}}\right],\quad \varphi\in B_b(\mathcal{O}),\; x\in\g,\; t>0
\end{equation} 
where $\mi{X}$ is the generalized mild solution of \eqref{eqFO}, and $\tau_x$ is the stopping time defined by
\begin{equation*}
\tau_x=\inf\{ s> 0\; : \; X(s,x)\in \g^c \}.
\end{equation*}  
We will prove that $\nu$ is sub-invariant for $P^{\mathcal{O}}(t)$; therefore $P^{\mathcal{O}}(t)$ is uniquely extendable to a strongly continuous semigroup $P_p^{\mathcal{O}}(t)$ in $L^p(\g,\nu)$, for $p\geq 1$. We denote by $M_2$ the infinitesimal generator of $P_2^\g(t)$. In Section \ref{gradsist} we will restrict to the case where $F$ is a gradient perturbation, namely it has a potential. In this case the invariant measure $\nu$ is a weighted Gaussian measure and it is possible to associate a quadratic form $\mathcal{Q}_2$ to $N_2$, see for example \cite{AD-CA-FE1, AN-FE-PA1, CAP-FER1, CAP-FER2,DA2,DA-LU2,DA-LU3,DA-LU-TU1,DA-TU1,DA-TU2,FER1}. Under some additional hypotheses (Hypotheses \ref{DIRI}) we will define the Sobolev space $W_{C}^{1,2}(\X,\nu)$, and we will show that there exists a quadratic form $\mathcal{Q}_2$ on $W_{C}^{1,2}(\X,\nu)$ such that
\[
\int_\X (N_2\varphi)\psi d\nu=\mathcal{Q}_2(\varphi,\psi)=-\frac{1}{2}\int\scal{C^{1/2}\D\varphi}{C^{1/2}\D\psi}d\nu,\quad \forall \varphi\in\Dom(N_2),\;\psi\in W_C^{1,2}(\X,\nu).
\]
After, proceeding as in \cite[Section 3]{DA-LU1}, we will consider a suitable Sobolev space $\mathring{W}_C^{1,2}(\X,\nu)$ of the functions $u:\g\ra\R$ such that their null extension $\widehat{u}$  belongs to $W_{C}^{1,2}(\X,\nu)$, and the quadratic form $\mathcal{Q}_2^\g$ on $\mathring{W}_{C}^{1,2}(\g,\nu)$ defined by
\[
\mathcal{Q}^\g_2(\varphi,\psi)=\mathcal{Q}_2(\widehat{\varphi},\widehat{\psi}),\quad \forall \varphi,\psi\in \mathring{W}_{C}^{1,2}(\g,\nu).
\]
In Subsection \ref{ingO} we are going to prove the last result of this paper.

\begin{thm}\label{identifDiri}
Assume that Hypotheses \ref{DIRI} hold true. Then the infinitesimal generator $M_2$ of $P^\g_2(t)$ is the operator $N^\g_2$ associated with $\mathcal{Q}^\g_2$, namely
\begin{align*}
\Dom(N^\g_2):=\{\varphi\in\mathring{W}_{C}^{1,2}(\mathcal{O},\nu)\; &:\; \exists \beta\in L^2(\g,\nu) \mbox{ s.t. } \int_{\g}\beta\psi d\nu=\mathcal{Q}_2^{\mathcal{O}}(\beta,\psi)\;\; \forall\psi\in \mathring{W}_{C}^{1,2}(\X,\nu) \}\\
&N^\g_2\varphi=\beta,\quad \varphi\in\Dom(N^\g_2).
\end{align*}

\end{thm}
This result generalizes the one contained in \cite[Section 3]{DA-LU1} proved for $F=0$. For a study of an analogous problem in the case where $\X$ is a separable Banach space and $F=0$ we refer to \cite{AS-VA1}, instead we refer to \cite{PRI1,TAL1} for other types of problems about the semigroup \eqref{semi2}.

To prove the above mentioned results, we rely on the fact that $\nu$ has finite moments of any order and $\nu(E)=1$. In \cite{DA-DE-GO1,DA-RO1,DA-RO-WA1} the authors assume as hypothesis the existence of a measure with the properties just mentioned, in this paper we show the existence and uniqueness of such a measure for the class of dissipative systems that we are considering. In \cite[Sections 7.2 and 11.6]{DA-ZA4} the authors prove existence and uniqueness of the generalized mild solution of \eqref{eqFO} and of the invariant measure for $P(t)$ in many settings that include our own. However they do not provide the estimates of the moments that we require. Instead in \cite[Chapters 4]{DA1} and \cite[Chapters 6-7]{CER1} the authors prove the estimates that we need, but in a specific context (in the same context, see \cite{CER4,CER5} for the case of multiplicative noise and \cite{CER-LU1} for the nonautonomous case). In particular they assume that $F$ is a Nemytskii operator; in Subsection \ref{noN} we will present an example of $F$ that it is not of this type. Then in Section \ref{GMS} we will prove that the SPDE \eqref{eqFO} has a unique generalized mild solution $\mi{X}$ and the transition semigroup $P(t)$ defined in \eqref{semiX} has a unique invariant measure $\nu$ with finite moments of any order, and $\nu(E)=1$.

We conclude this introduction comparing our assumptions and examples with the one already present in the literature.
Hypotheses \ref{EU3} include the dissipative case of the SPDE considered in \cite[Section 6]{CER1} (see \cite[Hypothesis 3]{CE-DA1}). In \cite[Section 6]{CER1} the authors set $\X=L^2(\Gamma,\lambda,\R^n)$ and $E=C(\overline{\Gamma},\R^n)$ where $n\in\N$, $\lambda$ is the Lebesgue measure and $\Gamma$ is an open set in $\R^d$ with $d\leq 3$. They consider as $F$ a Nemytskii type operator
\[
F(x)(\xi):=f(\xi,x(\xi))+cx(\xi),\quad x\in L^2(\Gamma,\lambda,\R^n),\; \xi\in\Gamma,\; c>0, 
\]
where $f:\Gamma\times \R^n\ra \R$ is a suitable function. This setting covers a large class of reaction-diffusion systems such as the one considered in \cite[Chapter 4]{DA1}. In Subsection \ref{gra} we are going to present a particular case of \cite[Section 6]{CER1} that verifies also the hypotheses of Theorem \ref{identifDiri}. In Subsection \ref{noN} we give an example of $F$ that satisfies the hypotheses of Theorem \ref{identif}, but it does not belong to the class considered in \cite[Chapter 6]{CER1} and \cite[Chapter 4]{DA1}. Finally in Subsection \ref{belapp} we will consider a particular case of the example of Subsection \ref{noN}, where the invariant measure $\nu$ is concentrated on $W^{1,2}([0,1],\lambda)$.

\section{Preliminaries}
In this Section we recall some notations, definitions and results that we will use in the rest of the paper.
\subsection{Notations}
Let $H_1$ and $H_2$ be two Banach spaces. We denote by $\mathcal{B}(H_1)$ the family of the Borel subsets of $H_1$ and by $B_b(H_1;H_2)$ the set of the $H_2$-valued, bounded and Borel measurable functions. When $H_2=\R$ we simply write $B_b(H_1)$. We denote by $C_b(H_1;H_2)$ the set of the continuous and bounded functions from $H_1$ to $H_2$. If $H_2=\R$ we simply write $C_b(H_1)$. We denote by $C^k_b(H_1;H_2)$, $k\in\N\cup\set{\infty}$ the set of the $k$-times Fr\'echet differentiable functions from $H_1$ to $H_2$ with bounded derivatives up to order $k$. If $H_2=\R$ we simply write $C_b^k(H_1)$. For a function $\Phi\in C_b^1(H_1;H_2)$ we denote by $\J \Phi(x)$ the Fr\'echet derivative operator of $\Phi$ at the point $x\in H_1$. Let $H_3$ be a Hilbert space, if $f\in C_b^1(H_3)$ then, for every $x\in H_3$ there exists a unique $k\in H_3$ such that for every $h\in H_3$
\[\J f(x)(h)=\gen{h,k}_{H_3}.\]
We set $\D f(x):=k$.
\begin{defi}
Let $G: \Dom(G)\subseteq H_1\ra H_1$  and let $H_2\subseteq H_1$. We call part of $G$ in $H_2$ the function $G_{H_2}:\Dom(G_{H_2})\subseteq H_2\ra H_2$ defined by
\[
\Dom(G_{H_2}):=\{x\in \Dom(G)\cap H_2\; :\; G(x)\in H_2  \},\quad G_{H_2}(x):=G(x),\; x\in \Dom(G_{H_2}).
\]
\end{defi}
We denote by $\mathcal{L}(H_1)$ the set of bounded linear operators from $H_1$ to itself and by $\Id_{H_1}\in\mathcal{L}(H_1)$ the identity operator on $H_1$. $\Gamma\in\mathcal{B}(H_1)$, we denote by $\mathbb{I}_\Gamma$ the characteristic function of $\Gamma$. We say that $B\in\mathcal{L}(H_1)$ is \emph{non-negative} (\emph{positive}) if for every $x\in H_1\setminus\set{0}$
\[
\gen{Bx,x}\geq 0\ (>0).
\]
In an anologous way we define the non-positive (negative) operators. We recall that a bounded and non negative operator is self-adjoint. Let $B\in\mathcal{L}(H_1)$ be a non-negative and self-adjoint operator. We say that $B$ is a trace class operator if
\begin{align*}
\Tr[B]:=\sum_{n=1}^{+\infty}\scal{Be_n}{e_n}<+\infty,
\end{align*}
for some (and hence for all) orthonormal basis $\{e_n\}_{n\in\N}$ of $H_1$. We recall that the trace is independent of the choice of the basis (see \cite[Section XI.6 and XI.9]{DUN-SCH2}).

Let $(\Omega,\mathcal{F},\{\mathcal{F}_t\}_{t\geq 0},\mathbb{P})$ be a filtered (normal) probability space and let $\K$ be a separable Banach space. Let $\xi:(\Omega,\mathcal{F},\mathbb{P})\ra (\K,\mathcal{B}(\K))$ be a random variable, we denote by 
\[
\mathscr{L}(\xi):=\mathbb{P}\circ\xi^{-1}
\] 
the law of $\xi$ on $(\K,\mathcal{B}(K))$, and by
\[
\mathbb{E}[\xi]:=\int_\Omega \xi(w)\  \mathbb{P}(d\omega)=\int_\K x\  \mathscr{L}(\xi)(dx)
\]
the expectation of $\xi$ with respect to $\mathbb{P}$. In this paper when we refer to a $\K$-valued process we mean an adapted process defined on $(\Omega,\mathcal{F},\{\mathcal{F}_t\}_{t\geq 0},\mathbb{P})$ with values in $\K$. Let $\{Y_1(t)\}_{t\geq 0}$ e $\{Y_2(t)\}_{t\geq 0}$ be two stochastic processes; we say that $\{Y_1(t)\}_{t\geq 0}$ is a version (or a modification) of $\{Y_2(t)\}_{t\geq 0}$ if, for any $t\geq 0$ we have
\[
Y_1(t)=Y_2(t),\quad \qc
\]
Let $\{Y(t)\}_{t\geq 0}$ be a $\K$-valued process we say that $\{Y(t)\}_{t\geq 0}$ is continuous if the map $Y(\cdot):[0,+\infty)\ra \K$ is continuous $\qc$.

We recall the definitions of two Banach spaces often considered in the literature (see \cite[Section 6.2]{CER1}).

\begin{defi}\label{spacep}$ $
\begin{enumerate}
\item Let $I$ be an interval contained in $\R$ and $p\geq 1$. We denote by $\K^p(I)$ the space of progressive measurable $\K$-valued processes $\{Y(t)\}_{t\in I}$ endowed with the norm
\[
\norm{\{Y(t)\}_{t\in I}}_{\K^p(I)}^p:=\sup_{t\in \overline{I}}\E[\norm{Y(t)}_\K^p].
\]

\item Let $I$ be an interval contained in $\R$ and $p\geq 1$. We denote by $C_p(I,\K)$ the space of continuous $\K$-valued processes $\{Y(t)\}_{t\in I}$ endowed with the norm
\[
\norm{\{Y(t)\}_{t\in I}}_{C_p(I,\K)}^p:=\E[\sup_{t\in \overline{I}}\norm{Y(t)}_\K^p].
\]
\end{enumerate}
\end{defi}

\subsection{Dissipative mappings}

We recall some basic results about subdifferential and dissipative maps, we refer to \cite[Appendix A]{CER1} and \cite[Appendix D]{DA-ZA4} for the results of this section.
Let $\mathcal{K}$ be a separable Banach space. For any $x\in\mathcal{K}$, we define the subdifferential $\partial \norm{x}_\K$ of $\norm{\cdot}_\K$ at $x\in\K$ as
\[
\partial \norm{x}:=\{ x^*\in E\; |\; \dscal{x}{x^*}{\mathcal{K}}=\norm{x}_{\mathcal{K}},\; \norm{x^*}_{\mathcal{K}^*}=1\}.
\]
Let $[t_0,t_1]\subset[0,+\infty)$ and let $u:[t_0,t_1]\ra\K$ be a differentiable function. Then the function $\gamma:=\norm{u}_\K:[t_0,t_1]\ra [0,+\infty)$ is left-differentiable in any $t_0\in [t_0,t_1]$ and
\begin{equation}\label{Ldiff}
\dfrac{d^-\gamma}{dt}(t_0):=\lim_{h\ra0^-}\dfrac{\gamma(t_0+h)-\gamma(t_0)}{h}=\min\{\dscal{u'(t_0)}{x^*}{E}\;:\; x^*\in\partial\norm{u(t_0)}_\K \}.
\end{equation}
Moreover, let $b\in\R$ and $g:[t_0,t_1]\ra [0,+\infty)$ be a continuous function, if
\[
\dfrac{d^-\gamma}{dt}(t)\leq b\gamma(t)+g(t), \quad t\in [t_0,t_1],
\]
then, for any $t\in [t_0,t_1]$, we have
\begin{equation}\label{varofcost}
\gamma(t)\leq e^{b(t-t_0)}\gamma(t_0)+\int^t_{t_0}e^{b(t-s)}g(s)ds,\quad t\in [t_0,t_1].
\end{equation} 

\begin{defi}\label{dissi}
A map $f:\Dom(f)\subseteq \mathcal{K}\ra\mathcal{K}$ is said to be dissipative if, for any $\alpha>0$ and $x,y\in \Dom(f)$, we have
\begin{equation}\label{disban1}
\norm{x-y-\alpha(f(x)-f(y))}_{\mathcal{K}}\geq \norm{x-y}_{\mathcal{K}} 
\end{equation}
If $f$ is a linear operator \eqref{disban1} reads as
\[
\norm{(\lambda\Id-A)x}_\K\geq \lambda \norm{x}_\K, \quad \forall \lambda>0,\; x\in\Dom(A)
\]
We say that $f$ is m-dissipative if the range of $\lambda\Id-f$ is all the space $\mathcal{K}$ for some $\lambda>0$ (and so for all $\lambda>0$).
\end{defi}

Using the notion of subdifferential we have the following useful charaterization for the dissipative maps.

\begin{prop}
Let $f:\Dom(f)\subseteq \mathcal{K}\ra\mathcal{K}$. $f$ is dissipative if and only if, for any $x,y\in \Dom(f)$ there exists $z^*\in\partial\norm{x-y}$ such that
\begin{equation}\label{disban}
\dscal{f(x)-f(y)}{z^*}{\mathcal{K}}\leq 0.
\end{equation}
If $\mathcal{K}$ is a Hilbert space \eqref{disban} becomes
\begin{equation*}
\scal{f(x)-f(y)}{x-y}_{\mathcal{K}}\leq 0.
\end{equation*}
\end{prop}

\subsection{Semigroups}\label{P3}
In this subsection we recall some basic definitions and results of the semigroups theory. We refer to \cite[Chapter 2]{LUN1} and \cite[Chapter II]{EN-NA1}.

Let $\K$ be a separable Banach space. Let $T(t)$ be a semigroup on $B_b(\K)$.
\begin{enumerate}
\item We say that $T(t)$ is non-negative if for any non-negative valued $\varphi\in B_b(\K)$ and for any $t\geq 0$, $T(t)\varphi$ has non-negative values.
\item We say that $T(t)$ is Feller, if for any $t\geq 0$ we have
\[
T(t)\left(C_b(\K)\right)\subseteq C_b(\K).
\]
\item We say that $T(t)$ is contractive, if for any $t\geq 0$ and $\varphi\in B_b(\K)$ we have
\[
\norm{T(t)\varphi}_{\infty}\leq \norm{\varphi}_{\infty}.
\]
\end{enumerate}   

\begin{defi}\label{defiinv}
Let $\mu\in\mathscr{P}(\K)$ (the set of all Borel probability measures on $\K$) we say that $\mu$ is invariant for $T(t)$ if, for any $\varphi\in C_b(\K)$ and $t\geq 0$, we have
\[
\int_\K T(t)\varphi(x)\nu(dx)=\int_\K \varphi(x)\nu(dx).
\]
\end{defi}

Let $B:\Dom(B)\subseteq \K\ra\K$. We denote by $\rho(B)$ the resolvent set of $B$ and for $\lambda\in\rho(B)$ we denote by $R(\lambda,B)$ the resolvent operator of $B$. 

We consider the complexification of $\K$, and we still denote it by $\K$. Let $B:\Dom(B)\subseteq\K\ra\K$ be a sectorial operator, namely there exist $M>0$, $\eta_0\in\R$ and $\theta_0\in (\pi/2,\pi]$ such that
\begin{equation*}
S_{0}:=\{\lambda\in\CC\; |\; \lambda\neq \eta_0,\; \vert \mbox{arg}(\lambda-\eta_0)\vert<\theta_0 \}\subseteq\rho(B);
\end{equation*}
\begin{equation}\label{eta0}
\norm{R(\lambda,B)}_{\mathcal{L}(\K)}\leq\frac{M}{\abs{\lambda-\eta_0}},\quad \forall\lambda\in S_0.
\end{equation}
We denote by $e^{tB}$ the analytic semigroup generated by $B$. We recall some basic properties:
\begin{enumerate}
\item there exists $M_0>0$ such that for any $t>0$
\begin{equation}\label{anlitic1}
\norm{e^{tB}}_{\mathcal{L}(\K)}\leq M_0e^{\eta_0t};
\end{equation}
\item for any $t>0$ and $h\in\N$
\begin{equation}\label{anlitic2}
e^{tB}(\K)\subseteq \Dom(B^h);
\end{equation}
\item for any $x\in\overline{\Dom(B)}$
\begin{equation}\label{anlitic3}
\lim_{n\ra +\infty}nR(n,B)x=x;
\end{equation} 
\item let $f(t)=e^{tB}$, we have 
\begin{equation}\label{anlitic4}
f\in C^{\infty}((0,+\infty),\mathcal{L}(\K)).
\end{equation} 
\end{enumerate}

\begin{rmk}
Properties analogous to \eqref{anlitic1} and \eqref{anlitic3} are also verified by strongly continuous semigroups. Moreover, for the strongly continuous semigroup, the function $f$ of \eqref{anlitic4} belongs to $C([0,+\infty),\mathcal{L}(\K))$. 
\end{rmk}

\subsection{The Ornstein--Uhlenbeck case}
We recall some results about the Ornstein--Uhlenbeck semigroups that will be used in Subsection \ref{core}. We assume that $F=0$ and that
\begin{equation*}
\int^t_0\Tr[e^{sA}Ce^{sA^*}]ds<+\infty,\quad\forall t\geq 0.
\end{equation*} 
The SPDE \eqref{eqFO} becomes
\begin{gather}\label{eq0}
\eqsys{
dZ(t,x)=AZ(t,x)dt+\sqrt{C}dW(t), & t>0;\\
Z(0,x)=x\in \X,
}
\end{gather}
where ${W(t)}_{t\geq 0}$ is a $\X$-valued cylindrical Wiener process. We refer to \cite[Section 4.1.2]{DA-ZA4} and \cite[Section 1]{PES-ZA1} for a definition of cylindrical Wiener process.
It is well known that $Z(t,x)=e^{tA}x+W_A(t)$ is the unique mild solution of \eqref{eq0} and, for any $t>0$, we have
\begin{equation*}
W_A(t):=\int^t_0e^{(t-s)A}\sqrt{C}dW(s)\sim\mathcal{N}(0,Q_t),\quad Q_tx=\int^t_0e^{tA}Ce^{tA^*}.
\end{equation*}
So, via a change of variable, for any $\varphi\in B_b(\X)$ we obtain
\begin{equation}\label{Koyuki}
T(t)\varphi(x):=\E[\varphi(Z(t,x))]=\int_{\Omega}\varphi(e^{tA}x+W_A(t))d\mathbb{P}=\int_\X \varphi(e^{tA}x+y)\mathcal{N}(0, Q_t)(dy).
\end{equation}
Now we consider the Banach space
\[C_{b,2}(\X):=\set{f:\X\ra \R \tc x\mapsto\frac{f(x)}{1+\norm{x}^2}\text{ belongs to } C_b(\X)}.\] 
endowed with the norm
\[\norm{f}_{b,2}:=\sup_{x\in\X}\left(\frac{\vert f(x)\vert}{1+\norm{x}^2}\right),\qquad f\in C_{b,2}(\X).\]
It is known that the semigroup $T(t)$ is not strongly continuous on $(C_{b,2}(\X), \norm{\cdot}_{b,2})$. For a detailed study of the semigroup $T(t)$, defined in \eqref{Koyuki}, in spaces of continuous functions with weighted sup-norms, we refer to \cite{CER2, CER3}, \cite[Section 2.8.3]{DA1} and \cite[Section 2]{DA-TU2}. Instead the semigroup $T(t)$ is strongly continuous on $C_{b,2}(\X)$ with respect the mixed topology. For an in-depth study of the mixed topology we refer to \cite{GO-KO1}; in the following theorems we list some properties that will be used.

\begin{thm}[Theorems 4.1, 4.2 and 4.5 of \cite{GO-KO1}]\label{KOKO}
$ $
\begin{enumerate}[\rm (i)]
\item The semigroup $T(t)$, introduced in \eqref{Koyuki}, is strongly continuous on $C_{b,2}(\X)$ with respect to the mixed topology. We denote by $(L_{b,2},\Dom(L_{b,2}))$ its infinitesimal generator.

\item For any $\lambda>0$, $\varphi\in C_{b,2}(\X)$ and $x\in\X$, we consider the integral
\begin{equation*}
J(\lambda)\varphi:=\int^{+\infty}_0e^{-\lambda t}T(t)\varphi dt.
\end{equation*}
For every $\lambda>0$, the operator
\[J(\lambda):(C_{b,2}(\X),\tau_M)\ra(C_{b,2}(\X),\tau_M)\]
is continuous (here $\tau_M$ denotes the mixed topology), and $J(\lambda)\varphi=R(\lambda,L_{b,2})\varphi$.

\item  $(L_{b,2},\Dom(L_{b,2}))$ is the closure of the operator $L_0$ (defined in \eqref{OPFOC}) to $C_{b,2}(\X)$, endowed with the mixed topology.

\end{enumerate}
\end{thm}
We remark that $L_{b,2}$ is the weak infinitesimal generator of the semigroup $T(t)$ on $C_{b,2}(\X)$ in the sense of \cite{CER2, CER3} (see \cite[Remark 4.3]{GO-KO1}). Finally we recall the following approximation result. 

\begin{prop}[Propositions 2.5 and 2.6 of \cite{DA-TU2}]\label{appxiA}
Let $\varphi\in \Dom(L_{b,2})\cap C^1_b(\X)$. There exists a family $\{\varphi_{n_1,n_2,n_3,n_4}\,|\,n_1,n_2,n_3,n_4\in\N\}\subseteq\xi_A(\X)$ such that for every $x\in\X$
\begin{align*}
\lim_{n_1\rightarrow+\infty}\lim_{n_2\rightarrow+\infty}\lim_{n_3\rightarrow+\infty}\lim_{n_4\rightarrow+\infty}&\varphi_{n_1,n_2,n_3,n_4}(x)=\varphi(x);\\
\lim_{n_1\rightarrow+\infty}\lim_{n_2\rightarrow+\infty}\lim_{n_3\rightarrow+\infty}\lim_{n_4\rightarrow+\infty}&\D\varphi_{n_1,n_2,n_3,n_4}(x)=\D\varphi(x);\\ 
\lim_{n_1\rightarrow+\infty}\lim_{n_2\rightarrow+\infty}\lim_{n_3\rightarrow+\infty}\lim_{n_4\rightarrow+\infty}&L_{b,2}\varphi_{n_1,n_2,n_3,n_4}(x)=L_{b,2}\varphi(x).
\end{align*}
Furthermore there exists a positive constant $C_\varphi$ such that, for any ${n_1,n_2,n_3,n_4}\in\N$ and $x\in\X$, it holds
\begin{equation}\label{Konatsu}
\vert\varphi_{n_1,n_2,n_3,n_4}(x)\vert+ \norm{\D\varphi_{n_1,n_2,n_3,n_4}(x)}+\vert L_{b,2}\varphi_{n_1,n_2,n_3,n_4}(x)\vert\leq C_{\varphi}(1+\norm{x}^2).
\end{equation}
\end{prop}
\noindent For a proof of the previous result we refer to \cite[Section 2.8.3]{DA1} or \cite[Section 2]{DA-TU2}. See also \cite[Section 8]{DA-LU-TU1}.

\section{The SPDE \eqref{eqFO}}\label{GMS}
In this section we will study the generalized mild solution $\mi{X}$ of the SPDE \eqref{eqFO} and the invariant measure of the transition semigroup $P(t)$ of \eqref{semiX}. In subsection \ref{mildinE} we will prove that, for any $x\in E$, the SPDE \eqref{eqFO} has a unique mild solution. In subsection \ref{GMS2} we will focus on the generalized mild solution of the SPDE \eqref{eqFO} and on the transition semigroup \eqref{semiX}. In subsection \ref{misinvS} we will prove that the semigroup $P(t)$ has a unique invariant measure $\nu$ and we will investigate some properties of $\nu$. 

We recall some useful inequalities that we will use frequently in this section.
\begin{equation}\label{youngyoung}
ab\leq \frac{(q-1)(\epsilon a)^{q/(q-1)}}{q}+\frac{(b/\epsilon)^{q}}{q},\quad \forall a,b,\epsilon>0,\; q>1,
\end{equation}
If $\K$ is a Banach space for every $h_1,h_2\in \K$ and $r\geq 1$ it holds
\begin{align}\label{Nina}
\norm{h_1-h_2}_\K^r\geq 2^{1-r}\norm{h_1}_\K^r-\norm{h_2}_\K^r.
\end{align}  

\subsection{The mild solution for $x$ belonging to $E$}\label{mildinE}

We state the hypotheses under which we work in this subsection.

\begin{hyp}\label{EU2}$ $

\begin{enumerate}[\rm(i)]

\item\label{EU2.1} There exists a Banach space $E\subseteq \Dom(F)$ Borel measurable, densely and continuously embedded in $\X$ such that $F(E)\subseteq E$.

\item\label{EU2.3} $A$ generates a strongly continuous and analytic semigroup $e^{tA}$ on $\X$ and $A_E$ (the part of $A$ in $E$) generates an analytic semigroup $e^{tA_E}$ on $E$.

\item\label{EU2.4} There exists $\zeta\in\R$ such that
\begin{enumerate}
\item $A+F-\zeta\Id$ is dissipative in $\X$;
\item $A_E+F_{|E}-\zeta\Id$ is dissipative in $E$.
\end{enumerate}

\item\label{EU2.5} $\{W_{A}(t)\}_{t\geq 0}$ is a $E$-valued continuous process such that
\begin{equation}\label{condl2}
\int^T_0\Tr[e^{sA}Ce^{sA^*}]ds<+\infty,\quad T>0.
\end{equation}

\item\label{EU2.7} There exist $M>0$ and $m\in\N$ such that 
\[
\norm{F(x)}_E\leq M(1+\norm{x}^m_E),\quad x\in E.
\]

\item\label{EU2.6} $F_{|E}:E\ra E$ is locally Lipschitz on $E$, namely $F_{|E}$ is Lipschitz continuous on bounded sets of $E$

\end{enumerate}
\end{hyp}

\begin{rmk}\label{rmkb}
$ $
\begin{enumerate}
\item Hypotheses \ref{EU2}\eqref{EU2.7} or \ref{EU2}\eqref{EU2.6} imply that $F_{|E}$ maps bounded sets of $E$ into bounded sets of $E$, and so, since $E$ is continuously embedded in $\X$, $F$ maps bounded sets of $E$ into bounded sets of $\X$.

\item Hypothesis \ref{EU2}\eqref{EU2.6} does not imply that $F:\Dom(\X)\subseteq\X\ra\X$ is continuous, however it implies that $F_{|E}:E\ra\X$ is continuous.
\end{enumerate}
\end{rmk}

\begin{rmk}\label{momconvu}
For any $T>0$, by \eqref{condl2} and \cite[Theorem 5.2]{DA-ZA4}, the process $W_{A,T}:=\{W_A(t)\}_{t\in [0,T]}$ can be seen as a $L^2([0,T],\lambda,\X)$-valued gaussian random variable, where $\lambda$ is the Lebsegue measure. Moreover by Hypotheses \ref{EU2}\eqref{EU2.5} and the same arguments used in \cite[Remark 3.4]{MAS1}, the process $W_{A,T}$ is a $C([0,T],E)$-valued gaussian random variable. Hence
\[
\E[\sup_{t\in [0,T]}\norm{W_{A}(t)}^p_{E}]<+\infty,\quad \forall p\geq 1,
\]
by Hypotheses \ref{EU2}\eqref{EU2.7} we have
\[
\E[\sup_{t\in [0,T]}\norm{F(W_{A}(t))}^p_{E}]<+\infty,\quad \forall p\geq 1,
\]
and, since $E$ is continuously embedded in $\X$, we obtain
\[
\E[\sup_{t\in [0,T]}\norm{F(W_{A}(t))}^p+\sup_{t\in [0,T]}\norm{W_{A}(t)}^p]<+\infty,\quad \forall p\geq 1.
\]
\end{rmk}

We recall the standard definition of mild solution.

\begin{defi}\label{Mild1}
For any  $x\in E$ we call mild solution of \eqref{eqFO} a $E$-valued process $\mi{X}$ such that, for any $t\geq 0$, we have
\begin{align}\label{mildF}
& X(t,x)(\omega)=e^{tA}x+\int_0^te^{(t-s)A}F(X(s,x)(\omega))ds+W_A(t)(\omega),\quad\qc,
\end{align}
Moreover we say that the mild solution $\mi{X}$ is unique if every process $\mi{Y}$ that verifies \eqref{mildF} then $\mi{Y}$ is a version of $\mi{X}$.
\end{defi}

To prove that, for any $x\in E$, the SPDE \eqref{eqFO} has a unique mild solution $\mii{X}$ we need to exploit an approximating problem. For simplicity, from here on we still denote by $A$ the part of $A$ in $E$. For any $x\in E$ and large $n\in\N$, we introduce the approximate problem
\begin{gather}\label{eqFOn}
\eqsys{
dX_n(t,x)=\big(AX_n(t,x)+F(X_n(t,x))\big)dt+RdW(t), & t>0;\\
X_n(0,x)=nR(n,A)x.
}
\end{gather}
\begin{rmk}\label{rmketa}
By Hypotheses \ref{EU2}(\ref{EU2.3}), $e^{tA}$ verifies \eqref{eta0} with  $\eta_0\in\R$ and $M_0>0$. So $R(n,A)$ is defined only for $n>\eta_0$, hence if $\eta_0>1$, then we consider \eqref{eqauxapp} for $n>\eta_0$.
\end{rmk}

Now we are going to prove that, for any $x\in E$ and large $n\in\N$ the SPDE \eqref{eqFOn} has unique mild solution $\mi{X_n}\in C_p([0,T],E)$, for any $p\geq 1$ and $T>0$ (see Definition \eqref{spacep}). To do this we consider the equation
\begin{gather}\label{eqauxapp}
\eqsys{
\dfrac{dY_n}{dt}(t,x)=AY_{n}(t,x)+F(Y_{n}(t,x)+W_{A}(t)), & t>0;\\
Y_{n}(0,x)=n R(n,A)x.
}
\end{gather}

If we show that, for any $x\in E$ and large $n\in\N$, equation \eqref{eqauxapp} has a unique mild solution $\mi{Y_n}\in C_p([0,T],E)$, for any $p\geq 1$ and $T>0$ (see Definition \eqref{spacep}), then, by Remark \ref{momconvu}, the process $\mi{X_n}$ defined by
\begin{equation}\label{processon}
X_n(t,x):=Y_n(t,x)+W_A(t),\quad \qc,
\end{equation}
is the unique mild solution of \eqref{eqFOn} in $C_p([0,T],E)$, for any $p\geq 1$ and $T>0$.

\begin{prop}\label{solapp}
Assume that Hypotheses \ref{EU2} hold true. For any $x\in E$ and large $n\in\N$ problem \eqref{eqauxapp} has a unique mild solution $\mi{Y_n}\in C_p([0,T],E)$,  for any $p\geq 1$ and $T>0$. Moreover there exists a sequence of processes $\{\{Y_{n,k}(t,x)\}_{t\geq 0}\}_{k\in\N}$ such that
\[
t\ra Y_{n,k}(t,x)\in C^1([0,T],E)\cap C([0,T],\Dom(A)),\quad \forall T>0,\;\forall k\in\N,\;\quad \qc
\]
\begin{align}\label{coco1}
\lim_{k\ra +\infty}\norm{Y_{n,k}(\cdot,x)-Y_{n}(\cdot,x)}_{C([0,T],E)}=0,\quad \lim_{k\ra +\infty}\norm{o_{n,k}(x)}_{C([0,T],E)}=0,\quad\forall\; T>0\; \qc
\end{align}
where
\begin{equation}\label{opiccolo}
o_{n,k}(t,x)=\dfrac{dY_{n,k}}{dt}(t,x)-AY_{n,k}(t,x)-F(Y_{n,k}(t,x)+W_{A}(t)),\quad \qc
\end{equation}
In addition for any $p\geq 1$ there exist $C_{p}:=C_{p}(\zeta)>0$ and $\kappa_{p}:=\kappa_{p}(\zeta)\in\R$ such that for any $x\in E$, large $n\in\N$ and $t>0$
\begin{align}
&\norm{Y_n(t,x)}^p\leq  C_p\left(e^{\kappa_p t}\norm{x}^p+\int_0^te^{\kappa_p (t-s)}\norm{F(W_A(s))}^p ds\right),\quad\qc\label{stindXn}\\
&\norm{Y_n(t,x)}_E^p
\leq  C_p\left(e^{\kappa_p t}\norm{x}^p_E+\int_0^te^{\kappa_p (t-s)}\norm{F(W_A(s))}^p_Eds\right)\quad\qc\label{stindEn}
\end{align}
\end{prop}

\begin{proof}
We prove the statements for a fixed large $n\in\N$ and $x\in E$. By Hypotheses \ref{EU2}\eqref{EU2.5}, the trajectories of the process $\{W_A(t)\}_{t\geq 0}$ are continuous $\qc$
In this proof we work pathwise, so we will denote by $w_A(\cdot)$ a fixed trajectory of $\{W_A(t)\}_{t\geq 0}$. We fix $T>0$ and we consider the equation
\begin{gather}\label{eqauxappdet}
\eqsys{
\dfrac{dy_n}{dt}(t,x)=Ay_{n}(t,x)+F(y_{n}(t,x)+w_{A}(t)), & t\in [0,T];\\
y_{n}(0,x)=n R(n,A)x,
}
\end{gather}
and the operator
\[
V(y)(t):=e^{tA}n R(n,A)x+\int^t_0e^{(t-s)A}F(y(s)+w_A(s))ds,\quad y\in C([0,T],E),\; t\in [0,T].
\]
Let $R>M_0\norm{x}_E\sup_{t\in [0,T]}e^{t\eta_0}$. By \eqref{anlitic1}, \eqref{anlitic4}, Remark \ref{rmketa} and the local lipshitzianity of $F$, for any $y,z\in  C([0,T],E)$ such that $\norm{y}_{C([0,T],E)}, \norm{z}_{C([0,T],E)}\leq R$ , we have
\[
\norm{V(y)}\leq M_0\norm{x}_E\sup_{t\in [0,T]}e^{t\eta_0}+M_0\sup_{t\in [0,T]}\norm{F(y(t)+w_A(t))}_E\sup_{t\in [0,T]}\int^t_0e^{(t-s)\eta_0}ds
\]
\[
\norm{V(y)-V(z)}_{C([0,T],E)}\leq L_R M_0\norm{y-z}_{C([0,T],E)}\sup_{t\in [0,T]}\int^t_0e^{(t-s)\eta_0}ds.
\]
where $M_0$ and $\eta_0$ are the constants in Remark \ref{rmketa} and $L_R>0$ is the Lipschitz constant of $F$ on the ball in $C([0,T],E)$ with center $0$ and radius $R$. By Remark \eqref{rmkb} for $T_0\in [0,T]$ small enough $V(B(0,R))\subseteq B(0,R)$ and $V$ is a contraction in $B(0,R)$ where $B(0,R)$ is the ball in $C([0,T_0],E)$ with center $0$ and radius $R$. Hence by the contraction mapping theorem the problem  \eqref{eqauxappdet} has a unique mild solution $y_{n,T_0}(\cdot,x)\in B(0,R)$. To prove that there exists a global solution $y_{n,T}$ of \eqref{eqauxappdet} in $C([0,T],E)$ it is sufficient to prove an estimate for $\norm{y_{n,T_0}(\cdot,x)}_{C([0,T_0],E)}$ independent of $T_0$.
By \cite[Proposition 4.1.8]{LUN1} $y_{n,T_0}(\cdot,x)$ is the strong solution of 
\begin{gather*}
\eqsys{
\dfrac{dv_{n}}{dt}(t,x)=Av_{n}(t,x)+F(y_{n,T_0}(t,x)+w_{A}(t)), & t\in [0,T_0];\\
v_{n}(0,x)=n R(n,A)x,
}
\end{gather*}
namely there exists a sequence $\{y_{n,k,T_0}(\cdot,
x)\}_{k\in\N}\subseteq C^1([0,T_0],E)\cap C([0,T_0],\Dom(A))$ such that
\begin{align}
&\lim_{k\ra +\infty}\norm{y_{n,k,T_0}(\cdot,x)-y_{n,T_0}(\cdot,x)}_{C([0,T_0],E)}=0,\notag\\
&\lim_{k\ra +\infty}\norm{\dfrac{dy_{n,k,T_0}}{dt}(\cdot,x)-Ay_{n,k,T_0}(\cdot,x)-F(y_{n,T_0}(\cdot,x)+w_{A}(\cdot))}_{C([0,T_0],E)}=0\label{coco2}.
\end{align}
For any $t\in [0,T_0]$, $x\in E$ and $n,k\in\N$ we set
\begin{equation*}
o_{n,k,T_0}(t,x)=\dfrac{dy_{n,k,T_0}}{dt}(t,x)-Ay_{n,k,T_0}(t,x)-F(y_{n,k,T_0}(t,x)+w_{A}(t)),
\end{equation*}
hence we have
\begin{align*}
\norm{o_{n,k,T_0}(t,x)}_E&\leq \norm{\dfrac{dy_{n,k,T_0}}{dt}(t,x)-Ay_{n,k,T_0}(t,x)-F(y_{n,T_0}(t,x)+w_{A}(t))}_E
\\&+\norm{F(y_{n,T_0}(t,x)+w_{A}(t))-F(y_{n,k,T_0}(t,x)+w_{A}(t))}_E\\
&\leq \norm{\dfrac{dy_{n,k,T_0}}{dt}(t,x)-Ay_{n,k,T_0}(t,x)-F(y_{n,T_0}(t,x)+w_{A}(t))}_E\\
&+L_R\norm{y_{n,T_0}(t,x)-y_{n,k,T_0}(t,x)}_E,
\end{align*}
and so, by \eqref{coco2}, for any large $n\in\N$ we obtain 
\begin{align*}
\lim_{k\ra +\infty}\norm{o_{n,k,T_0}(x)}_{C([0,T_0],E)}=0,\quad\; \qc
\end{align*}

Let $x\in E$, $p\geq 1$, $k,n\in\N$ and $t\in [0,T_0]$. By \eqref{Ldiff}-\eqref{opiccolo} and Hypotheses \ref{EU2}\eqref{EU2.4}, there exists $y^*\in\partial \norm{y_{n,k}(t,x)}_E$, such that
\begin{align}
\frac{1}{p}\dfrac{d^-\norm{y_{n,k,T_0}(t,x)}^p_E}{dt}&\leq\norm{y_{n,k,T_0}(t,x)}_E^{p-1}\dscal{Ay_{n,k,T_0}(t,x)}{y^*}{E}\notag\\
&+\norm{y_{n,k,T_0}(t,x)}_E^{p-1}\dscal{F(y_{n,k,T_0}(t,x)+w_A(t))}{y^*}{E}\notag\\
&+\norm{y_{n,k,T_0}(t,x)}_E^{p-1}\dscal{o_{n,k,T_0}(t,x)}{y^*}{E}\notag\\
&=\norm{y_{n,k,T_0}(t,x)}_E^{p-1}\dscal{Ay_{n,k,T_0}(t,x)}{y^*}{E}\notag\\
&+\norm{y_{n,k,T_0}(t,x)}_E^{p-1}\dscal{F(y_{n,k,T_0}(t,x)+w_{A}(t))-F(w_A(t))}{y^*}{E}\notag\\
&+\norm{y_{n,k,T_0}(t,x)}_E^{p-1}\dscal{F(w_A(t))}{y^*}{E}\notag\\
&+\norm{y_{n,k,T_0}(t,x)}_E^{p-1}\dscal{o_{n,k,T_0}(t,x)}{y^*}{E}\notag\\
&\leq \zeta\norm{y_{n,k,T_0}(t,x)}^p_E+\norm{y_{n,k,T_0}(t,x)}^{p-1}_E\left(\norm{F(w_A(t))}_E+\norm{o_{n,k,T_0}(t,x)}_E\right)\label{bhbh}.
\end{align}
We claim that there exists $C_1:=C_1(\zeta,p)$ such that
\begin{align}\label{C1}
&\frac{1}{p}\dfrac{d^-\norm{y_{n,k,T_0}(t,x)}^p_E}{dt}\leq C_1\norm{y_{n,k,T_0}(t,x)}^p_E+\frac{1}{p}\left(\norm{F(w_A(t))}_E+\norm{o_{n,k,T_0}(t,x)}_E\right)^p.
\end{align}
Indeed for $p=1$, \eqref{C1} is verified with $C_1=\zeta$, instead, for $p>1$, applying \eqref{youngyoung} in \eqref{bhbh} with $a=\norm{y_{n,k,T_0}(t,x)}^{p-1}_E$, $b=\left(\norm{F(w_A(t))}_E+\norm{o_{n,k,T_0}(t,x)}_E\right)$, $q=p$ and $\epsilon=1$ we obtain
\begin{align*}
&\frac{1}{p}\dfrac{d^-\norm{y_{n,k,T_0}(t,x)}^p_E}{dt}\leq (\zeta+\frac{p-1}{p})\norm{y_{n,k,T_0}(t,x)}^p_E+\frac{1}{p}\left(\norm{F(w_A(t))}_E+\norm{o_{n,k,T_0}(t,x)}_E\right)^p,
\end{align*}
and so \eqref{C1} is verified with $C_1=\zeta+\frac{p-1}{p}$. By \eqref{varofcost}, \eqref{eta0}, Remark \ref{rmketa} and \eqref{C1} we get 
\begin{equation*}
\norm{y_{n,k,T_0}(t,x)}^{p}_E\leq e^{pC_1 t}\norm{x}_E^p+\int_0^te^{pC_1 (t-s)}(\norm{F(w_A(t))}_E+\norm{o_{n,k,T_0}(t,x)}_E)^pds,
\end{equation*}
and letting $k\ra+\infty$, by \eqref{coco1},				
\begin{equation}\label{preep}
\norm{y_{n,T_0}(t,x)}^{p}_E\leq e^{pC_1 t}\norm{x}_E^p+\int_0^te^{pC_1 (t-s)}\norm{F(w_A(t))}_E^pds.
\end{equation}
By Remark \ref{momconvu} and recalling that $T_0\in [0,T]$, for any $t>0$ we obtain 
\begin{equation}\label{preep1}
\norm{y_{n,T_0}(t,x)}^{p}_E\leq \norm{x}_E^p+\frac{1}{pC_1}(e^{pC_1t}-1)\sup_{t\in [0,T]}\norm{F(w_A(t))}_E^p.
\end{equation}
and so there exists a global solution $y_{n,T}$ of \eqref{eqauxappdet} in $C([0,T],E)$. The uniqueness of $y_{n,T}$ follows immediately by \eqref{preep1}, the local lipschitzianity and the Gronwall inequalities of $F$. We have proved that, for any $T>0$ the equation \eqref{eqauxapp}, has unique mild solution $y_{n,T}\in C([0,T],E)$. We consider the continuous function $y_n(\cdot,x):[0,+\infty)\ra E$ defined by
\[
y_n(\cdot,x)_{|[0,T]}=y_{n,T}(\cdot,x),\quad \forall\; T>0.
\]
Exploiting \cite[Proposition 4.1.8]{LUN1} (as we have already do for $y_{n,T_0}$) for any $T>0$, there exists a sequence  $\{y_{n,k,T_0}(\cdot,
x)\}_{k\in\N}\subseteq C^1([0,T],E)\cap C([0,T],\Dom(A))$ such that 
\begin{align*}
\lim_{k\ra +\infty}\norm{y_{n,k}(\cdot,x)-y_{n}(\cdot,x)}_{C([0,T],E)}=0,\quad \lim_{k\ra +\infty}\norm{o_{n,k}(x)}_{C([0,T],E)}=0,\quad\forall\; \qc
\end{align*}
where
\begin{equation*}
o_{n,k}(t,x)=\dfrac{dy_{n,k}}{dt}(t,x)-Ay_{n,k}(t,x)-F(y_{n,k}(t,x)+w_{A}(t)),\quad \qc
\end{equation*}
Moreover $y_n(\cdot,x)$ verifies \eqref{preep}, for any $p\geq 1$ and $t>0$. The process $\mi{Y_n}$ whose trajectories are the functions $y(\cdot,x)$ verifies the statements of the proposition. Uniqueness follows by \eqref{preep}, local lipschitzianity of $F$ and the Gronwall inequality. Estimates \eqref{stindXn} follows in exactly the same way as \eqref{stindEn} using the inner product of $\X$ instead of the duality product of $E$ and $E^*$. 
\end{proof}

\begin{rmk}
If $e^{tA}$ is strongly continuous also on $E$, then it is possible to replace the initial datum $nR(n,A)x$ by $x$, in \eqref{eqFOn}.
\end{rmk}


By Remark \ref{momconvu}, \eqref{Nina} (with $h_1=X_n(t,x)$, $h_2=W_A(t)$ and $r=p$) and Proposition \ref{solapp} we obtain immediately the following result.

\begin{prop}\label{proyos}
Assume that Hypotheses \ref{EU2} hold true. For any large $n\in\N$ and $x\in E$ the process $\mi{X_n}$, defined in \eqref{processon}, is the unique mild solution of \eqref{eqFOn} in $C_p([0,T],E)$, for any $p\geq 1$ and $T>0$. In addition, for any $p\geq 1$, large $n\in\N$, $x\in E$ and $t>0$, we have
 \begin{align}
&\norm{X_n(t,x)}^p\leq  C'_p\left(e^{\kappa_p t}\norm{x}^p+\int_0^te^{\kappa_p (t-s)}\norm{F(W_A(s))}^p ds +\norm{W_A(t)}^p\right),\;\qc\label{stindXzn}\\
&\norm{X_n(t,x)}_E^p
\leq  C'_p\left(e^{\kappa_p t}\norm{x}^p_E+\int_0^te^{\kappa_p (t-s)}\norm{F(W_A(s))}^p_Eds+\norm{W_A(t)}_E^p\right),\;\qc\label{stindEzn}
\end{align}
where $C'_{p}:=\max\left(2^{p-1}C_p,2^{p-1}\right)$, and $C_p$, $\kappa_{p}$ are the constants of Proposition \ref{solapp}.
\end{prop}

Now we prove a convergence result for $\mi{X_n}$.

\begin{thm}\label{mcyos}
Assume that Hypotheses \ref{EU2} hold true. For any  $x\in E$, there exists\\ $\mi{X}\in C_p((0,T],E)\cap C_p([0,T],\X)$, for any $p\geq 1$ and $T>0$, such that
\begin{equation}\label{convpuntyos}
\lim_{n\ra+\infty}\norm{X_n(\cdot,x)-X(\cdot,x)}_{C([0,T],\X)}=0,\quad \forall\; T>0,\;\qc,
\end{equation}
\begin{equation}\label{covergenzayosE}
\lim_{n\ra\infty}\norm{X_{n}(\cdot,x)-X(\cdot,x)}_{C([\epsilon,T],E)}=0,\quad \forall\; 0<\epsilon\leq T,\; \qc
\end{equation}
For any $p\geq 1$, let $C'_{p}$ be the constant of Proposition \ref{proyos} and let $\kappa_{p}$ be the constants of Proposition \ref{solapp}. For any $p\geq 1$, $x\in E$ and $t>0$, we have
 \begin{align}
&\norm{X(t,x)}^p\leq  C'_p\left(e^{\kappa_p t}\norm{x}^p+\int_0^te^{
\kappa_p (t-s)}\norm{F(W_A(s))}^p ds +\norm{W_A(t)}^p\right),\;\qc\label{stindXz}\\
&\norm{X(t,x)}_E^p
\leq  C'_p\left(e^{\kappa_p t}\norm{x}^p_E+\int_0^te^{\kappa_p (t-s)}\norm{F(W_A(s))}^p_Eds+\norm{W_A(t)}_E^p\right),\;\qc\label{stindEz}
\end{align}
Moreover there exists a constant $\eta\in\R$ such that, for any $x,y\in E$ and $t>0$, we have 
\begin{align}
&\norm{X(t,x)-X(t,y)}\leq e^{\eta t}\norm{x-y},\quad\qc,\label{lipdete}
\end{align}
\begin{align}\label{lipdeteE}
\norm{X(t,x)-X(t,y)}_E\leq e^{\eta t}\norm{x-y}_E,\quad\qc
\end{align}
\end{thm}

\begin{proof}
As in the proof of Proposition \ref{solapp} we work pathwise, so we denote by $y_{n,k}(\cdot,x)$, $y_n(\cdot,x)$ and $w_A(\cdot)$ fixed trajectories of the process $\mi{Y_{n,k}}$, $\mi{Y_{n}}$ and $\{W_A(t)\}_{t\in [0,T]}$ respectively.
We begin to prove \eqref{convpuntyos} for a fixed $T>0$. Let $x\in E$, $k,n\in\N$, $t\in [0,T]$. We define
\[
z_{n,k}(t,x):=y_{n,k}(t,x)+w_A(t),\quad n,k\in\N.
\]
We stress that $z_{n,k}(t,x)-z_{m,k}(t,x)=y_{n,k}(t,x)-y_{m,k}(t,x)$, for any $n,m\in\N$. For any $n,m\in\N$, by \eqref{opiccolo}, we have
\begin{align*}
\frac{1}{2}\dfrac{d\norm{z_{n,k}(t,x)-z_{m,k}(t,x)}^2}{dt}&\leq\scal{A(z_{n,k}(t,x)-z_{m,k}(t,x))}{z_{n,k}(t,x)-z_{m,k}(t,x)}\\
&+\scal{F(z_{n,k}(t,x))-F(z_{m,k}(t,x))}{z_{n,k}(t,x)-z_{m,k}(t,x)}\\
&+\scal{o_{n,k}(t,x)-o_{m,k}(t,x)}{z_{n,k}(t,x)-z_{m,k}(t,x)},
\end{align*}
by Hypotheses \ref{EU2}\eqref{EU2.4}  we have
\begin{align*}
\phantom{aaa}\frac{1}{2}\dfrac{d\norm{z_{n,k}(t,x)-z_{m,k}(t,x)}^2}{dt}&\leq \zeta\norm{z_{n,k}(t,x)-z_{m,k}(t,x)}^2\\
&+\norm{o_{n,k}(t,x)-o_{m,k}(t,x)}\norm{z_{n,k}(t,x)-z_{m,k}(t,x)}.
\end{align*}
 By \eqref{youngyoung} (with $\epsilon=1$ and $q=2$) we have
\begin{align*}
\frac{1}{2}\dfrac{d\norm{z_{n,k}(t,x)-z_{m,k}(t,x)}^2}{dt}&\leq (\zeta+\frac{1}{2})\norm{z_{n,k}(t,x)-z_{m,k}(t,x)}^2+\frac{1}{2}(\norm{o_{n,k}(t,x)}+\norm{o_{m,k}(t,x)})^2.
\end{align*}
We set $H_1=\zeta+\frac{1}{2}$, by \eqref{varofcost} we obtain
\begin{align*}
\norm{z_{n,k}(t,x)-z_{m,k}(t,x)}^2 &\leq e^{2H_1t}\norm{(nR(n,A)-mR(m,A))x}\\
& +\int^t_0e^{-2H_1(t-s)}(\norm{o_{n,k}(t,x)}+\norm{o_{m,k}(t,x)})^2ds.
\end{align*}
Letting $k\ra+\infty$, by \eqref{coco1} and Remark \ref{rmkb} we have
\begin{align*}
\norm{z_{n}(t,x)-z_{m}(t,x)}^2 &\leq e^{2H_1t}\norm{(nR(n,A)-mR(m,A))x}.
\end{align*}
where 
\[
z_n(t,x)=X_n(t,x)(w)=y_n(t,x)+w_A(t).
\]
By \eqref{anlitic3}, we obtain that, for any $T>0$ and $x\in E$, the sequence $\{z_n(\cdot,x)\}_{n\in\N}$ is a Cauchy sequence in $C([0,T],\X)$ and we denote by $z_T(\cdot,x)\in C_b([0,T],\X)$ its limit. A continuous function $z(\cdot,x):[0,+\infty)\ra \X$ such that
\begin{equation}\label{incollo}
z(\cdot,x)_{|[0,T]}:= z_T(\cdot,x),\quad \forall\; T>0,
\end{equation} 
is well defined. So the process $\mi{X}$, whose trajectories are the functions $z(\cdot,x)$, verifies \eqref{convpuntyos}. \eqref{stindXzn} and \eqref{convpuntyos} yields \eqref{stindXz} and, by Remark \ref{momconvu} and \eqref{stindXz}, we have $\mi{X}\in C_p([0,T],\X)$ for any $p\geq 1$ and $T>0$.

Now we prove \eqref{covergenzayosE} for fixed $\epsilon, T>0$. By \eqref{stindEzn}, for any $x\in E$, there exists $R:=R(x,T)>0$ such that for any large $n\in\N$ and $t\in [\epsilon,T]$ we have
\begin{equation*}
\norm{z_n(t,x)}_E\leq R.
\end{equation*}
Let $L:=L(x,T)>0$ be the Lipschitz constant of $F$ on $B_E(0,R)$. So, for any $x\in E$, $n,m\in\N$ and $t\in [\epsilon,T]$, by \eqref{anlitic2} Remark \ref{rmketa} and the local lipschitzianity of $F$ we have
\begin{align*}
\norm{z_n(t,x)-z_m(t,x)}_E&\leq \norm{(nR(n,A)-mR(m,A))e^{tA}x}_E\\
&+M_0L\int^t_0e^{t\eta_0}\norm{z_n(s,x)-z_m(s,x)}_Eds
\end{align*}
Hence, by the Gronwall inequality, there exists $K_2:=K_2(x,T)>0$ such that
\begin{equation}\label{prenc}
\norm{z_n(t,x)-z_m(t,x)}_E\leq K_2\norm{(nR(n,A)-mR(m,A))e^{tA}x}_E.
\end{equation}
Letting $m,n\ra +\infty$ in \eqref{prenc}, by \eqref{anlitic3} we obtain that, for any $T>0$, $\epsilon>0$ and $x\in E$, the sequence $\{z_n(\cdot,x)\}_{n\in \N}$ is a Cauchy sequence in $C([\epsilon,T],E)$ and, since $E$ is continuously embedded in $\X$, its limit is the same in $C([0,T],\X)$. So the function defined in \eqref{incollo} is continuous from $(0,+\infty)$ to $E$ and the process $\mi{X}$, which verifies \eqref{convpuntyos}, verifies also \eqref{covergenzayosE}. \eqref{stindEzn} and \eqref{covergenzayosE} yields \eqref{stindEz}, and by Remark \ref{momconvu} and \eqref{stindEz}, we have $\mi{X}\in C_p((0,T],E)$.

Now we prove \eqref{lipdete}. Let $T>0$ and $x,y\in E$. For any, $t\in [0,T]$, $k,n\in\N$, by \eqref{opiccolo} we have 
\begin{align*}
&\frac{1}{2}\dfrac{d\norm{z_{k,n}(t,x)-z_{k,n}(t,y)}^2}{dt}\leq \scal{A(z_{k,n}(t,x)-z_{k,n}(t,y))}{z_{k,n}(t,x)-z_{k,n}(t,y)}\\
&+\scal{F(z_{k,n}(t,x)+w_A(t))-F(z_{k,n}(t,y)+w_A(t))}{z_{k,n}(t,x)-y_{k,n}(t,y)}\\
&+\scal{o_{k,n}(t,x)-o_{k,n}(t,y)}{z_{k,n}(t,x)-z_{k,n}(t,y)}
\end{align*}
and by Hypotheses \ref{EU2}\eqref{EU2.4} we obtain
\begin{align*}
\frac{1}{2}\dfrac{d\norm{z_{k,n}(t,x)-z_{k,n}(t,y)}^2}{dt}\leq & \zeta\norm{z_{k,n}(t,x)-z_{k,n}(t,y)}^2
\\
&+\norm{o_{k,n}(t,x)-o_{k,n}(t,y)}\norm{z_{k,n}(t,x)-z_{k,n}(t,y)}.
\end{align*}
 By \eqref{youngyoung}(with $\epsilon=1$ and $q=2$) we have
\begin{equation}\label{L1}
\dfrac{d\norm{z_{k,n}(t,x)-z_{k,n}(t,y)}^2}{dt}\leq 2\eta\norm{z_{k,n}(t,x)-z_{k,n}(t,y)}^2+\frac{1}{2}\norm{o_{k,n}(t,x)-o_{k,n}(t,y)}^2,
\end{equation}
where $\eta=\zeta+\frac{1}{2}$. By \ref{varofcost} and letting $k\ra +\infty$ we obtain
\[
\norm{z_{n}(t,x)-z_{n}(t,y)}^2\leq e^{2\eta t}\norm{x-z}^2.
\]
Taking the square root and letting $n\ra+\infty$, by \eqref{convpuntyos} we obtain
\[
\norm{z(t,x)-z(t,y)}^2\leq e^{2\eta t}\norm{x-y}^2, \quad t\in [0,T],\; x,y\in E
\]
for any $T>0$ and for $\mathbb{P}$-a.a trajectory of $\mi{X}$, so \eqref{lipdete} is verified. Finally \eqref{lipdeteE} follows from \eqref{covergenzayosE} using similar arguments.
\end{proof}

We make some remarks about possible variations of Theorem \ref{mcyos}.

\begin{cor}\label{strongdis}
If the constant $\zeta$ in Hypotheses \ref{EU2}\eqref{EU2.4} is negative, then the constants $\kappa_p$ and $\eta$ are negative.
\end{cor}
\begin{proof}
Applying \eqref{youngyoung} with $\epsilon=\zeta$ if $\zeta\in (0,1]$, or with $\epsilon=1/\zeta$ if $\zeta>1$, we obtain that the constants $C_1$ of \eqref{C1} and $\eta$ of \eqref{L1} are negative.
\end{proof}

\begin{rmk}\label{strongnonanalit}
It is possible to require that the part of $A$ in $E$ generates a strongly continuous semigroup instead of an analytic one, in this case we can take $\epsilon=0$ in \eqref{covergenzayosE} and $\mi{X}\in C_p([0,T],E)$, for any $p\geq 1$ and $T>0$. 
\end{rmk} 

Let $x\in E$ and $\mi{X}$ be the process defined in Theorem \ref{mcyos}. Now we prove that it is the unique mild solutions of\eqref{eqFO}.

\begin{thm}\label{solMild}
Assume that Hypotheses \ref{EU2} hold true. For any $x\in E$, the process $\mi{X}$ is the unique mild solution of the SPDE \eqref{eqFO} in  $C_p([0,T],\X)\cap C_p((0,T],E)$, for any $p\geq 1$ and $T>0$. 
\end{thm}

\begin{proof}
We begin to prove uniqueness. Let $x\in E$ and let\\ $\mi{X_1},\mi{X_2}\in C_p((0,T],E)$, for any $p\geq 1$ and $T>0$, be two mild solution of \eqref{eqFO}. For any $0<t\leq T$, by Remark \ref{rmketa}, we have 
\[
\norm{X_1(t,x)-X_2(t,x)}_E\leq M_0\int^t_0 e^{(t-s)\eta_0}\norm{F(X_1(t,x))-F(X_2(t,x))}_Eds,\quad \qc
\]
Since $\mi{X_1},\mi{X_2}\in C_p((0,T],E)$, with $p\geq 1$, then
\[
\sup_{t\in [0,T]}\norm{X_1(t,x)}_E, \sup_{t\in [0,T]}\norm{X_2(t,x)}_E<+\infty,\quad \qc
\]
so by the local lipschitzianity of $F$, there exists $L:=L(x,T)>0$ such that
\[
\norm{X_1(t,x)-X_2(t,x)}_E\leq M_0L\int^t_0 e^{(t-s)\eta_0}\norm{X_1(t,x)-X_2(t,x)}_Eds,\quad \qc
\]
and by the Gronwall inequality we obtain
\[
X_1(t)=X_2(t),\quad\qc
\]
for any $t\in [0,T]$ and $T>0$, and so we have the uniqueness.

Now we prove that, for any $x\in E$, the process $\mi{X}$ is the mild solution of \eqref{eqFO}. Let $T>0$ and large $n\in\N$. We recall that, for any $t\in [0,T]$, we have 
\[
X_n(t,x):=Y_n(t,x)+W_A(t),\quad \qc
\]
hence, by Proposition \ref{solapp}
\begin{align}\label{pingpong2}
X_n(t,x)=e^{tA}nR(n,A)x+\int^t_0e^{(t-s)A}F(X_n(s,x))ds+W_A(t),\quad \qc
\end{align}
By \eqref{anlitic1}, Remarks \ref{rmkb}-\ref{rmkb}, \eqref{stindEz}, \eqref{covergenzayosE} and the dominated convergence theorem, we have
\[
\lim_{n\ra+\infty}\norm{\int^t_0 e^{(t-s)A}\left( F(X_n(s,x))-F(X(s,x)) \right) ds}_E=0,\quad \qc,
\]
so, letting $n\ra+\infty$ in \eqref{pingpong2}, by \eqref{anlitic3} we have
\[
X(t,x)=e^{tA}x+\int^t_0e^{(t-s)A}F(X(s,x))ds+W_A(t)\quad\qc
\]
for any $t\in [0,T]$ and $T>0$.
\end{proof}

\subsection{Generalized mild solution and transition semigroup}\label{GMS2}

Now we exploit the density of $E$ in $\X$ to define a process $\mi{X}$ for any $x\in\X$.

\begin{prop}\label{limmild}
Assume that Hypotheses \ref{EU2} hold true. For any $x\in\X$ there exists a unique process $\mi{X}\in C_p([0,T],E)$, for any $p\geq 1$ and $T>0$, such that
\begin{align}\label{convpungm}
&\lim_{n\ra+\infty}\norm{X(\cdot,x_n)-X(\cdot,x)}_{C([0,T],\X)}=0,\quad\forall\; T>0,\;\qc,
\end{align}
where $\{x_n\}_{n\in\N}\subseteq E$ converges to $x$ and $\{X(t,x_n)\}$ is the unique mild solution of \eqref{eqFO} with initial datum $x_n$. In addition, for any $p\geq 1$, $x,y\in\X$ and $t>0$,  we have 
{\small\begin{align}
&\norm{X(t,x)}^p\leq  C'_p\left(e^{\kappa_p t}\norm{x}^p+\int_0^te^{\kappa_p (t-s)}\left(\norm{F(W_A(s))}^p+\norm{W_A(s)}^p\right)ds +\norm{W_A(t)}^p\right),\;\qc,\label{stindXzX}\\
&\norm{X(t,x)-X(t,y)}\leq e^{\eta t}\norm{x-y},\quad\qc,\label{lipdeteX}
\end{align}}
where $\kappa_p$ is the constant of Proposition \ref{solapp}, $C'_p$ is the constant of Proposition \ref{proyos} and $\eta$ is the constant of Theorem \ref{mcyos}. Moreover, for any $x\in\X$, $p\geq 1$ and $T>0$ we have
\begin{align}
\lim_{n\ra+\infty}\norm{\{X(t,x_n)\}_{t\geq 0}-\mi{X}}^p_{C_p([0,T],\X)}=0\label{uhuhuh}.
\end{align}
\end{prop}

\begin{proof}
\noindent Since $E$ is dense in $\X$, for any $x\in\X$ there exists a sequence $\{x_n\}_{n\in \N}\subseteq E$ such that
\[
\lim_{m\ra+\infty} \norm{x_m-x}=0.
\]
We consider the sequence of mild solutions $\{\{X(t,x_m)\}_{t\in [0,T]}\}_{m\in\N}\subseteq C_p([0,T],\X)$, for any $p\geq 1$ and $T>0$, given by Theorem \ref{solMild}. Hence we have 
\[
\{\{X(t,x_n)\}_{t\geq 0}\}_{n\in\N}\subseteq C([0,T],\X),\quad\qc
\]
Moreover by \eqref{lipdete}, for any $T>0$ and $n_1,n_2\in\N$, we have
\[
\lim_{n_1,n_2\ra+\infty}\norm{X(\cdot,x_{n_1})-X(\cdot,x_{n_2})}_{C([0,T],\X)}=0,\quad\qc
\]
So there exists a unique $\X$-valued continuous process $\mi{X}$ (see Definition \ref{spacep}) that verifies \eqref{convpungm}. By \eqref{convpungm} the process $\mi{X}$ verifies \eqref{stindXzX}, \eqref{lipdeteX} and, by Remark \ref{momconvu},\\ $\mi{X}\in C_p([0,T],\X)$, for any $p\geq 1$ and $T>0$. Finally \eqref{lipdeteX} yields \eqref{uhuhuh}.
\end{proof}

\begin{defi}\label{genmild}
For any $x\in\X$ we call generalized mild solution of \eqref{eqFO} the limit \\$\mi{X}$ of Corollary \ref{limmild}. 
\end{defi}

Until now we have shown that
\begin{enumerate}
\item for any $x\in E$ the SPDE \eqref{eqFO} has a unique mild solution $\mi{X}\in C_p((0,T],E)\cap C_p([0,T],\X)$, for any $p\geq 1$ and for any $T>0$, in the sense of Definition \ref{Mild1};
\item for any $x\in \X$ the SPDE \eqref{eqFO} has a unique generalized mild solution $\mi{X}\in C_p([0,T],\X)$, for any $p\geq 1$ and for any $T>0$, in the sense of Definition \ref{genmild}, in particular if $x\in E$ then the generalized mild solution of \eqref{eqFO} is the mild solution of \eqref{eqFO}.
\end{enumerate}

So we can define the following families of operators.
\begin{defi}\label{defsem}
For every $t>0$ we set
\[
P(t)\varphi(x):=\E[\varphi(X(t,x))]=\int_\Omega\varphi(X(t,x)(\omega))\mathbb{P}(d\omega)\quad \varphi\in B_b(\X), x\in \X,
\]
where $\mi{X}$ is the unique generalized mild solution of \eqref{eqFO}.
Similarly we set  
\[
P^E(t)\varphi(x):=\E[\varphi(X(t,x))]=\int_\Omega\varphi(X(t,x)(\omega))\mathbb{P}(d\omega)\quad \varphi\in B_b(E), x\in E,
\]
where $\mi{X}$ is the unique mild solution of \eqref{eqFO}.
\end{defi}

By the same arguments of \cite{DA-ZA4}[Proposition 9.14 and Corollary 9.15] and taking into account \eqref{lipdeteE} and \eqref{lipdeteX}, we have the following result.

\begin{prop}\label{dimsemi}
$\{P(t)\}_{t\geq 0}$ and $\{P^E(t)\}_{t\geq 0}$ are two contraction, positive and Feller semigroups on $B_b(X)$ and $B_b(E)$ respectively.
\end{prop}

\subsection{Existence and Uniqueness of the invariant measure}\label{misinvS}

In this Subsection we are going to prove that the semigroup $P(t)$ has a unique invariant measure $\nu$ verifying some useful properties. To do this we need an additional hypothesis.

\begin{hyp}\label{EU3}
Assume that Hypotheses \ref{EU2} hold true. Moreover we assume that the constant $\zeta$ in Hypotheses \ref{EU2}\eqref{EU2.4} is negative and that
\begin{equation}\label{condinv}
\sup_{t\geq 0}\E[\norm{W_A(t)}^p]<+\infty,\quad\forall\; p\geq 1.
\end{equation}
\end{hyp} 
By Hypotheses \ref{EU2}\eqref{EU2.7} and \ref{EU3}\eqref{condinv} we have
\begin{equation*}
\Sigma_{p,E}:=\sup_{t\geq 0}\E[\norm{F(W_A(t))}_E^p+\norm{W_A(t)}_E^p]<+\infty,\quad\forall\; p\geq 1,
\end{equation*}
and, since $E$ is continuously embedded in $\X$, we have
\begin{equation*}
\Sigma_{p,\X}:=\sup_{t\geq 0}\E[\norm{F(W_A(t))}^p+\norm{W_A(t)}^p]<+\infty,\quad\forall\; p\geq 1.
\end{equation*}
For any $p\geq 1$ we set 
\begin{equation}\label{condinf}
\Sigma_p:=\max\{\Sigma_{p,\X},\Sigma_{p,E}\}.
\end{equation}
Hence by Corollary \ref{strongdis}, \eqref{stindEz} and \eqref{stindXzX} we obtain the following result.

\begin{prop}\label{stiinf}
Assume that Hypotheses \ref{EU3} hold true and let $\mi{X}$ be the generalized mild solution of \eqref{eqFO}. If $x\in \X$ then $\mi{X}\in \X^p([0,\infty))$, for any $p\geq 1$, if $x\in E$ then then $\mi{X}\in E^p([0,\infty))$, for any $p\geq 1$ (see Definition \ref{spacep}). In particular, for any $p\geq 1$, there exists $K_{p}:=K_{p}(\Sigma_{p},C'_p)$ (where $C'_p$ is the constant of Theorem \ref{proyos}), such that
\begin{align}
&\E[\norm{X(t,x)}^p]\leq K_p(1+e^{\kappa_pt}\norm{x}^p),\quad \forall\; t>0,\forall\; x\in\X,\label{stiXinf}\\
&\E[\norm{X(t,x)}_E^p]\leq K_p(1+e^{\kappa_pt}\norm{x}^p_E),\quad \forall\; t>0,\forall\; x\in E.\notag
\end{align}
where $\kappa_p<0$ is the constant of Proposition \ref{solapp} and Corollary \eqref{strongdis}.
\end{prop}

By \cite[Theorem 2.1]{DY1} and Proposition \ref{dimsemi}, it is possible to associate to the semigroups $P(t)$ and $P^E(t)$ two Markov transition functions, so we can exploit the results contained in \cite[Chapter 11]{DA-ZA4}.

\begin{thm}\label{invX}
Assume that Hypotheses \ref{EU3} hold true. There exists $\nu\in \mathscr{P}(E)$ such that it is the unique invariant measure (see Definition \ref{defiinv})) of both semigroups $P^E(t)$ and $P(t)$. Moreover $\nu(E)=1$ and it verifies the following properties,
\begin{align}\label{momeinvX}
\int_\X\norm{x}^p\nu(dx)<+\infty,\quad\forall\; p\geq 1,
\end{align} 
\begin{align}\label{momeinv}
\int_E\norm{x}_E^p\nu(dx)<+\infty,\quad\forall\; p\geq 1.
\end{align}
Moreover we have 
\begin{align}\label{madiainvX}
\lim_{t\ra +\infty}P(t)\varphi(x)=\int_\X\varphi(y)\nu(dy),\quad \varphi\in C_b(\X),\; x\in \X,
\end{align}
\begin{align}\label{madiainv}
\lim_{t\ra +\infty}P(t)\varphi(x)=\int_E\varphi(y)\nu(dy),\quad \varphi\in C_b(E),\; x\in E.
\end{align}
\end{thm}

\begin{proof}
Existence and uniqueness of the invariant measures $\nu$ and $\nu^E$ of $P(t)$ and $P^E(t)$ respectively follow by the same arguments of \cite[Theorems 11.33-11.34]{DA-ZA4}. However we write a sketch of the proof because it is useful to know how the invariant measure is constructed to prove \eqref{momeinvX}, \eqref{momeinv}, \eqref{madiainvX} and \eqref{madiainv}. 

Since $P^E(t)$ is Feller, by \cite[Propositions 11.1-11.4 and Remark 11.6]{DA-ZA4}, if there exists $\nu^E\in\mathscr{P}(E)$ such that, for any $x\in E$, $\mathscr{L}(X(t,x))$ narrow (or weak) converges to $\nu^E$, then $\nu^E$ is the unique invariant measure of $P^E(t)$. We recall that $\mathscr{L}(X(t,x))$ narrow (or weak) converges to $\nu^E$ if, for any $\varphi\in C_b(\K)$, we have
\begin{equation}\label{limmisinv}
\lim_{n\ra+\infty}\left|
\int_E \varphi(y)\mathscr{L}(X(t,x))(dy)-\int_E \varphi(y)\nu^E(dy)\right|=0.
\end{equation}
To prove \eqref{limmisinv} we consider the SPDE \eqref{eqFO} but with an arbitrary $s\in\R$ as initial time. Let $\{W'(t)\}_{t\geq 0}$ be another $\X$-valued cylindrical Wiener process independent of $\mii{W}$. For any $t\in\R$ we define the process 
\[
\widehat{W}(t):=\begin{cases}
W(t) & t\geq 0\\
W'(-t) & t< 0.
\end{cases}
\]
For any $s\in \R$ and $x\in\X$, we consider the SPDE
\begin{gather}\label{eqINV}
\eqsys{
dX(t,s,x)=\big(AX(t,s,x)+F(X(t,s,x))\big)dt+Cd\widehat{W}(t),\quad t\geq s\\
X(s,s,x)=x,
}
\end{gather}
We remark that the method used to prove Theorem \ref{solMild} and to define the generalized mild solution (see Corollary \ref{limmild})  also works by replacing the initial time $0$ by an arbitrary $s\in \R$.
Hence, for any $x\in \X$ and $s\in \R$, the SPDE \eqref{eqINV} has a unique generalized mild solution $\{X(t,s,x)\}_{t\geq 0}$. Moreover, as in Proposition \ref{stiinf}, for any $p\geq 1$ we have
\begin{align}
&\E[\norm{X(t,s,x)}_E^p]\leq K_p(1+e^{\kappa_p (t-s)}\norm{x}^p_E),\quad t\geq s, x\in \X\notag\\
&\E[\norm{X(t,s,x)}_E^p]\leq K_p(1+e^{\kappa_p (t-s)}\norm{x}^p_E),\quad t\geq s, x\in E\label{stindEMI}\\
&\E[\norm{X(t,s,x)-X(t,s,z)}]\leq e^{\eta(t-s)}\norm{x-z},\quad t\geq s, x,z\in\X\notag\\
&\E[\norm{X(t,s,x)-X(t,s,z)}_E]\leq e^{\eta(t-s)}\norm{x-z},\quad t\geq s, x,z\in E\label{lipmildI}
\end{align}
where $\kappa_p$ is the constant of Proposition \ref{solapp}, $\eta$ is the constant of Proposition \ref{mcyos} and $K_p$ is the constant of Proposition \ref{stiinf}. By Corollary \ref{strongdis} the constants $\eta$ and $\kappa_p$ are negative.

Now we prove that there exists a random variable $\xi\in L^2((\Omega,\mathbb{P}),E)$, such that, for any $x\in E$, we have
\begin{equation}\label{converinva}
\lim_{s\ra +\infty}\E\left[\norm{X(0,-s,x)-\xi}_E^2\right]=0,
\end{equation}
and after we will prove that the law of $\xi$ is the $\nu^E$ that verifies \eqref{limmisinv}. 
We can assume that $\{X(t,s,x)\}_{t\geq s}$ is a strict solution of \eqref{eqINV}, otherwise we approximate it as in Proposition \ref{solapp}. For $\mathbb{P}$-a.a. $\omega\in\Omega$, for any $x\in E$, $s\in \R$, $t\geq s$ and $h\in [s,t]$, by Hypotheses \ref{EU2}\eqref{EU2.4}, there exist $z^*\in \partial (\norm{X(t,s,x)(\omega)-X(t,h,x)(\omega)}_E)$ such that
\begin{align*}
\frac{1}{2}\frac{d\norm{X(t,s,x)-X(t,h,x)}_E}{dt}&=\dscal{A(X(t,s,x)-X(t,h,x))}{z^*}{E}\\
&+\dscal{F(X(t,s,x))-F(X(t,h,x))}{z^*}{E}\\
&\leq \zeta\norm{X(t,s,x)(\omega)-X(t,h,x)(\omega)}_E.
\end{align*}
By \ref{varofcost}, taking the square root and the expectation we obtain
\begin{align*}
\E\left[\norm{X(t,s,x)-X(t,h,x)}_E^2\right]&\leq e^{4\zeta(t-h)}\E\left[\norm{X(h,s,x)-x}_E^2\right]\notag\\
&\leq 2e^{4\zeta(t-h)}(\E\left[\norm{X(h,s,x)}_E^2\right]+\norm{x}_E^2),
\end{align*}
and so
\begin{equation}\label{precau}
\E\left[\norm{X(t,s,x)-X(t,h,x)}_E^2\right]\leq e^{4\zeta(t-h)}C_x,
\end{equation}
where $C_x:=2\sup_{r\geq s}\left(\E\left[\norm{X(r,s,x)}_E^2\right]\right)+2\norm{x}_E^2$ is finite by \eqref{stindEMI}. For any $x\in\X$, by Hypotheses \ref{EU3} and \eqref{precau}, the family $\{X(0,-t,x)\}_{t\geq 0}$ is Cauchy in $L^2((\Omega,\mathbb{P}),E)$, namely
\[
\lim_{s,t\ra +\infty}\E\left[\norm{X(0,-t,x)-X(0,-s,x)}_E^2\right]=0
\]
Since $L^2((\Omega,\mathbb{P}),E)$ is complete, then $\{X(0,-t,x)\}_{t\geq 0}$ converges in $L^2((\Omega,\mathbb{P}),E)$ and by \eqref{lipmildI} its limit does not depend on $x$, so \eqref{converinva} is verified. Let $\nu^E=\mathscr{L}(\xi)$, where $\xi$ is the random variable that verifies \eqref{converinva}. We prove that it verifies \eqref{limmisinv}. Since $\{W'(t)\}_{t\geq 0}$ and $\{W(t)\}_{t\geq 0}$ are two independent cylindrical Wiener processes; they have the same law, and so, for any $x\in E$ and $t\geq 0$, we have
\[
\mathscr{L}(X(t,x))=\mathscr{L}(X(0,-t,x)).
\]
Let $\varphi\in C_b(E)$. For any $x\in\X$, $t\geq 0$, we have
\begin{align*}
\int_\X\varphi(y)p_t(x,dy)&=\int_\X\varphi(y)\mathscr{L}(X(t,x))(dy)=\int_\X\varphi(y)\mathscr{L}(X(0,-t,x))(dy)\\
&=\int_\Omega\varphi(X(0,-t,x)(\omega))\mathbb{P}(d\omega).
\end{align*}
Since $\varphi\in C_b(E)$, by \eqref{converinva} and the dominated convergence theorem we have
\begin{equation}\label{prenar}
\lim_{t\ra+\infty}\int_\X\varphi(y)p_t(x,dy)=\lim_{t\ra+\infty}\int_\Omega\varphi(X(0,-t,x)(\omega))\mathbb{P}(d\omega)=\int_\Omega\varphi(\xi(\omega))\mathbb{P}(d\omega)=\int_\X\varphi(y)\nu^E(dy),
\end{equation}
hence \eqref{limmisinv} is verified and so the measure $\nu^E$ is the unique invariant measure of the transition semigroup $P^E(t)$. \eqref{madiainv} follows immediately by the definition of transition semigroup $P^E(t)$ and \eqref{prenar}. Now we prove \eqref{momeinv}. For $p\geq 1$ and $b>0$ we have
\[
\int_E \dfrac{\norm{y}_E^{p}}{1+b\norm{y}_E^{p}}p_{t}(x,dy)\leq \int_E \norm{y}_E^{p}p_{t}(x,dy)=\mathbb{E}[\norm{X(t,x)}_E^{p}].
\]
Then, by \eqref{stiXinf}, \eqref{madiainv} and the monotone convergence theorem, we conclude
\[
\int_\X\norm{y}_E^{p}\nu^E(dy)=\lim_{b\rightarrow 0}\lim_{t\rightarrow+\infty}\int_E \dfrac{\norm{y}_E^{p}}{1+b\norm{y}_E^{p}} p_{t}(x,dy)<+\infty.
\]
In the same way, we can prove that the semigroup $P(t)$ has a unique invariant measure $\nu$ that verifies \eqref{momeinvX} and \eqref{madiainvX}. By Hypotheses \ref{EU2}\eqref{EU2.1}, \cite[Lemma 2.1.1]{BOGIE2} and recalling that, for any $f:\X\ra\R$ and $A,B\subset\R$, $f^{-1}(A\cap B)= f^{-1}(A)\cap f^{-1}(B)$, it easy to prove that $B_b(\X)$ is continuously embedded in $B_b(E)$.  We recall that $B_b(E)$ and $B_b(\X)$ generate $\mathcal{B}(E)$ and $\mathcal{B}(\X)$ respectively. Hence $\mathcal{B}(\X)\subseteq \mathcal{B}(E)$ and the measure
\[
\nu'(\Gamma)=\nu^E(\Gamma\cap E),\quad  \Gamma\in\mathcal{B}(\X),
\]
is well defined. Moreover $\nu'$ is invariant for the semigroup $P(t)$, and by uniqueness, $\nu=\nu'$ and the statements of the theorem are verified.
\end{proof}

\begin{rmk}
 Theorem \ref{invX} implies that $\nu(H)=1$, for any $H\subseteq \X$ that verifies Hypotheses \ref{EU3} (see the example of subsection \ref{belapp} ).
\end{rmk}
\begin{rmk}
In some specific settings it is possible to prove Theorem \ref{invX} replacing the condition $\zeta>0$ with some other hypotheses on $F$ (e.g. \cite[Chapter 8]{CER1}).
\end{rmk}

\section{Behavior in $L^2(\X,\nu)$}\label{misinv}
First of all we show that the transition semigroup $P(t)$ is uniquely extendable to a strongly continuous and contraction semigroup in $L^2(\X,\nu)$. By the H\"older inequality and the invariance of $\nu$, for any $\varphi\in B_{b}(\X)$ we have
\begin{equation}\label{pcontra}
\norm{P(t)\varphi}_{L^2(\X,\nu)}^2=\int_\X \abs{P(t)\varphi}^2 d\nu \leq \int_\X P(t)\abs{\varphi}^2d\nu=\int_\X\abs{\varphi}^2d\nu=\norm{\varphi}_{L^2(\X,\nu)}^p.
\end{equation}
We recall that $C_b(\X)$ is dense in $L^2(\X,\nu)$. Observe that if $\{\varphi_n\}_{n\in\N}\subseteq C_b(\X)$ converges to $\varphi$ in $L^2(\X,\nu)$, then for any $t\geq 0$ the sequence $\{P(t)\varphi_n\}_{n\in\N}$ is Cauchy in $ L^2(\X,\nu)$. Indeed, by \eqref{pcontra}, we have
\[
\norm{P(t)\varphi_n-P(t)\varphi_m}_{L^2(\X,\nu)}\leq \norm{\varphi_n-\varphi_m}_{L^2(\X,\nu)}.
\]
Hence the transition semigroup $P(t)$ is uniquely extendable to a strongly continuous and contraction semigroup $P_2(t)$ in $L^2(\X,\nu)$. 

\begin{defi}
We denote by $N_2$ the infinitesimal generator of $P_2(t)$.
\end{defi}

\begin{rmk}
In a similar way, it is possible to prove that the semigroup $P^E(t)$ is uniquely extendable to a strongly continuous semigroup $P_2^E(t)$ in $L^2(E,\nu)$. In the rest of this paper we will not study $P_2^E(t)$ but only $P_2(t)$. However it is possible to prove a result analogous to Theorem \ref{identif} for $P_2^E(t)$ (see \cite{CE-DA1}). 
\end{rmk}

In this section we are going to prove Theorem \ref{identif}. As a first step in the next subsection we will study the behavior of $N_2$ on the set
\[
\xi_A(\X):=\Span\{\mbox{real and imaginary parts of the functions } x\mapsto e^{i\scal{x}{h}}\,|\, h\in \Dom(A^*)\}.
\]

\subsection{Behavior on $\xi_A(\X)$}

We recall the definition of $N_0$.
\begin{equation*}
 N_{0}\varphi(x):=\frac{1}{2}\tr[C\D^2\varphi(x)]+\scal{x}{A^*\D\varphi(x)}+\scal{F_0(x)}{\D\varphi(x)},\quad \varphi\in \xi_A(\X),\;x\in\X,
\end{equation*}
where
\begin{equation*}
F_0(x)=\begin{cases}
F(x) & x\in E,\\
0 & x\in \X\backslash E.
\end{cases}
\end{equation*}

 Let $\varphi\in\xi_A(\X)$, then there exist $m,n\in\N$, $a_1,\ldots,a_m,b_1,\ldots,b_n\in \R$ and\\ $h_1,\ldots,h_m,k_1,\ldots,k_n\in A^* $ such that
\[\varphi(x)=\sum_{i=1}^m a_i\sin(\gen{x,h_i})+\sum_{j=1}^n b_j\cos(\gen{x,k_j}).\]
Easy computations give for $x\in\X$
\begin{align*}
N_0\varphi(x)=\sum_{i=1}^m a_i&\pa{\gen{x,Ah_i}+\gen{F_0(x),h_i}-\frac{1}{2}\|C^{1/2}h_i\|}\sin(\gen{x,h_i})\notag\\
&\phantom{0000000}+\sum_{j=1}^n b_j\pa{\gen{x,Ak_j}+\gen{F_0(x),k_j}-\frac{1}{2}\|C^{1/2}k_j\|}\cos(\gen{x,k_j}),
\end{align*}
moreover by Hypothesis \ref{EU2}\eqref{EU2.7} and Theorem \ref{invX} we have
\begin{equation}\label{EU4.1}
\int_\X\norm{F_0(x)}^p d\nu(x)<+\infty,\quad \forall\; p\geq 1,
\end{equation}
and so $N_0\varphi$ belongs to $L^2(\X,\nu)$.

\begin{prop}\label{inc0}
Assume that Hypotheses \ref{EU3} hold true. $N_0$ is closable in $L^2(\X,\nu)$ and its closure $\overline{N}_0$ is dissipative in $L^2(\X,\nu)$. Moreover $N_2$ is an extension of $\overline{N}_0$, namely $\Dom(\overline{N}_0)\subseteq\Dom(N_2)$ and 
\begin{equation}\label{esten1}
\overline{N}_0\varphi=N_2\varphi,\quad \varphi\in \Dom(\overline{N}_0).
\end{equation}
\end{prop}

\begin{proof}
By Theorem \ref{solMild}, for any $x\in E$, the trajectories of $\mii{X}$ take values in $E$. So by \cite[Proof of Theorem 3.19]{DA1}, for any $\varphi\in\xi_A(\X)$ and $x\in E$, we have
\begin{align}
P_2(t)\varphi(x)=\E[\varphi(X(t,x))]&=\varphi(x)+\E\left[\int^t_0N_0\varphi(X(s,x))ds\right]\notag\\
&=\varphi(x)+\int^t_0 P(s)N_0\varphi(X(s,x))ds\label{sviE},
\end{align}
and so
\begin{equation}\label{vit1}
\lim_{t\rightarrow 0}\frac{P_2(t)\varphi(x)-\varphi(x)}{t}=N_0\varphi(x).
\end{equation}
To obtain \eqref{esten1} we need to prove that
\begin{align}
\lim_{t\rightarrow 0}\int_\X &\abs{\frac{P_2(t)\varphi(x)-\varphi(x)}{t}-N_0\varphi(x)}^2\nu(dx)=0,\quad \forall\;\varphi\in\xi_A(\X).\label{vit2}
\end{align}
We recall the Vitali convergence theorem (see \cite[Theorem 2.24]{FO-LE1}): \eqref{vit2} is verified if and only if the following three conditions are verified.
\begin{enumerate}
\item $\{\frac{P(t)\varphi-\varphi}{t}\}_{t\geq 0}$ converges in measure to $N_0\varphi$ with respect to the measure $\nu$.
\item For any $\eps>0$ there exists $\Gamma\in\mathcal{B}(\X)$ such that $\nu(\Gamma)<+\infty$ and
\begin{align*}
\frac{1}{t^2}\int_{(\X-\Gamma)}\vert P(t)\varphi(x)-\varphi(x))\vert^2  \nu(dx)\leq \eps\quad \forall\; t>0.
\end{align*}
\item For any $\eps>0$ there exists $\delta>0$ such that whenever $\Gamma\in\mathcal{B}(\X)$ with $\nu(\Gamma)<\delta$ we have
\begin{align*}
\frac{1}{t^2}\int_\Gamma\vert P(t)\varphi(x)-\varphi(x))\vert^2  \nu(dx)\leq \eps \quad \forall\; t>0.
\end{align*}
\end{enumerate}
By \eqref{vit1} and $\nu(E)=1$, (1) is verified. Since $\nu$ is a probability measure then (3) implies (2). We prove (3). We fix $\eps>0$. Since $N_0\varphi\in L^2(\X,\nu)$, there exists $\delta>0$ such that whenever $\Gamma\in\mathcal{B}(\X)$ with $\nu(\Gamma)<\delta$, then
\[\int_\Gamma|N_0\varphi(x)|^2\nu(dx)<\eps.\]
Recalling that $\nu(E)=1$, by the H\"older inequality, the invariance of $P(t)$ with respect to $\nu$ and \eqref{sviE} we have
\begin{align*}
\frac{1}{t^2}\int_\Gamma\vert P(t)\varphi(x)-\varphi(x))\vert^2  \nu(dx)&=\frac{1}{t^2}\int_{\Gamma\cap E}\vert P(t)\varphi(x)-\varphi(x))\vert^2 \nu(dx)\\
&=\int_{\Gamma\cap E}\abs{\int_0^tP(s)N_0\varphi(x)\frac{ds}{t}}^2\nu (dx)\\
&\leq \frac{1}{t}\int_0^t\left(\int_{\Gamma\cap E}\abs{P(s)(N_0\varphi)(x)}^2\nu (dx)\right)ds\\
&\leq  \frac{1}{t}\int_0^t\left(\int_{\Gamma\cap E} P(s)(\vert N_0\varphi\vert^2)(x)\nu (dx)\right)ds\\
&=  \frac{1}{t}\int_0^t\left(\int_{\Gamma\cap E}\vert N_0\varphi(x)\vert^2\nu (dx)\right)ds=\frac{1}{t}\int_0^t\eps ds=\eps.
\end{align*}
\noindent Hence, by the Vitali convergence theorem, we obtain \eqref{vit2} and so \eqref{esten1}. In particular, since $\nu$ is the invariant measure of $P_2(t)$, for any $\varphi\in\xi_A(\X)$, we have
\begin{equation}\label{invinf}
\int_\X N_0\varphi d\nu=\int_E N_2\varphi d\nu=0.
\end{equation}
Let $\varphi\in\xi_A(\X)$, by standard calculations we obtain
\[
N_0\varphi^2(x)=2\varphi(x)N_0\varphi(x)+\norm{C^{1/2}\D \varphi(x)}^2.
\]
Hence integrating with respect to $\nu$ and exploiting \eqref{invinf} we have
\[
\int_\X(N_0\varphi(x))\varphi(x)\nu(dx)=-\frac{1}{2}\int_\X\norm{C^{1/2}\D \varphi(x)}^2\nu(x),\quad\forall\;\varphi\in\xi_A(\X),
\]
so $N_0$ is dissipative and, since $\xi_A(\X)$ is dense in $L^2(\X,\nu)$, it is closable in $L^2(\X,\nu)$ and its closure $\overline{N}_0$ is dissipative in $L^2(\X,\nu)$.
\end{proof}

We conclude this subsection with a useful criterium to check whether a function $\varphi:\X\ra\R$ belongs to $\Dom(\overline{N}_0)$.
\begin{lemm}\label{inc}
Assume that Hypotheses \ref{EU3} hold true. If $\varphi\in\Dom(L_{b,2})\cap C^1_b(\X)$, then $\varphi\in\Dom(\overline{N}_0)$ and
\begin{equation*}
\overline{N}_0\varphi(x)=L_{b,2}\varphi(x)+\scal{F_0(x)}{\D\varphi(x)},\qquad x\in\X;
\end{equation*}
where $L_{b,2}$ is the operator introduced in Theorem \ref{KOKO}.
\end{lemm}
\begin{proof}
By Proposition \ref{appxiA} a family $\{\varphi_{n_1,n_2,n_3,n_4}\,|\, n_1,n_2,n_3,n_4\in\N\}\subseteq\xi_A(\X)$ exists such that, for any $x\in\X$,
\[
\lim_{n_1\rightarrow+\infty}\lim_{n_2\rightarrow+\infty}\lim_{n_3\rightarrow+\infty}\lim_{n_4\rightarrow+\infty}\overline{N_0}\varphi_{n_1,n_2,n_3,n_4}(x)=L_{b,2}\varphi(x)+\scal{F(x)}{\D\varphi(x)}.
\]
whenever $\varphi\in\Dom(L_{b,2})\cap C^1_b(\X)$. By \eqref{Konatsu}, there exists a constant $C_\varphi$, such that for any $x\in E$
\[
\vert \overline{N_0}\varphi_{n_1,n_2,n_3,n_4}(x)\vert=\vert N_0\varphi_{n_1,n_2,n_3,n_4}(x)\vert\leq C_\varphi(1+\norm{x}^{m+2})(1+\norm{F(x)}^2),
\]
so, since $\nu(E)=1$, by \eqref{EU4.1}, \eqref{momeinvX} and the Dominated Convergence theorem we obtain the statement.
\end{proof}

Before proving Theorem \ref{identif} we need to introduce an additional hypotheses which will allow us to use the regularizing sequence of $F$ defined in the next subsection.
\begin{hyp}\label{EU3var}
Assume that Hypotheses \ref{EU3} hold true and that there exists a constant $\zeta_2\in\R$ such that $F-\zeta_2\Id_\X:\Dom(F)\subset\X\ra\X$ is m-dissipative.
\end{hyp}

\subsection{A regularizing family for F}\label{reF}

First of all we consider the Yosida approximantions of $F$. For any $x\in\X$ and $\delta>0$, we consider the equation
\begin{equation}\label{eq-yos}
y-\delta (F(y)-\zeta_2 y)=x.
\end{equation}
By Hypothesis \ref{EU3var} the function $G(x):=F(x)-\zeta_2 x:\Dom(F)\subseteq\X\ra\X$ is m-dissipative, so by \cite[Proposition 5.5.3]{DA-ZA2} or \cite[Proposition A.2.2]{CER1} the equation \eqref{eq-yos} has a unique solution $x_\delta\in \Dom(F)$ and, for any $\delta>0$, the function $G_\delta(x):=G(x_\delta)$ verifies 
\begin{align}
&\norm{G_\delta(x)-G_\delta(z)}\leq \frac{2}{\delta}\norm{x-z},\quad x,z\in\X\label{1yy};\\
&\scal{G_\delta(x)-G_\delta(z)}{x-z}\leq 0, \quad x,z\in\X\label{2yy};\\
&\norm{G_\delta(x)}\leq \norm{G(x)},\quad x\in\X\label{3yy}.
\end{align}
Moreover for any $x\in\X$, we have
\begin{equation}\label{4yy}
\lim_{\delta\ra 0}\norm{G_\delta(x)-x}=0.
\end{equation}
For any $\delta>0$ we define
\[
F_\delta(x):=F(x_\delta)=G_\delta(x)-\zeta_2x_\delta.
\]
By \eqref{1yy} it follows immediately that $F_\delta$ is Lipschitz continuous and by \eqref{eq-yos} and \eqref{2yy} we have 
\begin{align*}
0\geq& \gen{G(x_\delta)-G(z_\delta),x-z}=\gen{F(x_\delta)-F(z_\delta),x-z}-\zeta_2\gen{x_\delta-z_\delta,x-z}\\
=& \gen{F(x_\delta)-F(z_\delta),x-z}+\zeta_2\gen{\delta(F(z_\delta)-\zeta_2 z_\delta)-\delta(F(x_\delta)-\zeta_2 x_\delta)+z-x,x-z}\\
=& \gen{F(x_\delta)-F(z_\delta),x-z}+\zeta_2\delta\gen{F(z_\delta)-F(x_\delta),x-z}\\
&\phantom{aaaaaaaaaaaa00000000000000000000}+\zeta_2^2\delta \gen{x_\delta-z_\delta,x-z}+\zeta_2 \gen{z-x,x-z}\\
=& \gen{F(x_\delta)-F(z_\delta),x-z}-\zeta_2\norm{x-z}^2+\zeta_2\delta\gen{G(z_\delta)-G(x_\delta),x-z}\\
\geq & \scal{F_{\delta}(x)-F_{\delta}(z)}{x-z}-\zeta_2\norm{x-z}^2,
\end{align*}
and so \(\scal{F_{\delta}(x)-F_{\delta}(z)}{x-z}\leq \zeta_2\norm{x-z}^2\). Moreover, for any $\delta>0$ and $x\in\X$, by \eqref{3yy} we have
\begin{align*}
\norm{F_\delta(x)}-\norm{F(x)}\leq \norm{F_\delta(x)-F(x)}\leq & \norm{G(x_\delta)-G(x)}+\zeta_2\norm{x_\delta-x}\notag\\
\leq &\norm{G(x_\delta)-G(x)}+\zeta_2\delta\norm{G(x_\delta)} \leq (2+\zeta_2\delta)\norm{G(x)},
\end{align*}
so
\begin{align}\label{3.5yy}
\norm{F_\delta(x)}\leq (2+\zeta_2\delta)\norm{G(x)}+\norm{F(x)}.
\end{align}
By \eqref{4yy} we have
\begin{equation}\label{5yy}
\lim_{\delta\ra 0}\norm{F_\delta(x)-x}=0.
\end{equation}
\noindent The approximations $\{F_\delta\}_{\delta>0}$ are Lipschitz continuous, however we need to introduce a further approximations that are at least Gateaux differentiable. For every $\delta,s>0$ and $x\in\X$, we define 
\[
F_{\delta,s}(x):=\int_\X F_{\delta}(y)\mathcal{N}(e^{-(s/2)Q^{-1}}x,Q_s)(dy),
\]
\noindent where $Q_s=Q(\Id-e^{-sQ^{-1}})$ and $Q$ is a positive and trace class operator on $\X$. This type of regularization is classical and it is based on the Mehler formula for the Ornstein--Uhlenbeck semigroup (see \cite[Proof of Theorem 11.2.14]{DA-ZA1}). By standard calculations, for any $\delta,s>0$ and $x,z\in\X$, we have that $F_{\delta,s}(x)$ is Lipschitz continuous and
\begin{align}
&\scal{F_{\delta,s}(x)-F_{\delta,s}(z)}{x-z}\leq \zeta_2\norm{x-z}^2.\notag
\end{align}
Moreover for any $\delta>0$ and $x\in\X$ we have
\begin{equation}\label{6yy}
\lim_{s\ra 0}\norm{F_{\delta,s}(x)-x}=0.
\end{equation}
For any $\delta>0$, $F_\delta$ is Lipschitz continuous so
\begin{equation}\label{7yy}
\sup_{s\geq 0}\sup_{x\in\X}\norm{F_{\delta,s}(x)}<+\infty.
\end{equation}
Hence by \eqref{3.5yy}, \eqref{5yy}, \eqref{6yy}, \eqref{7yy} and the dominated convergence theorem, we have
\begin{equation}\label{555}
\lim_{\delta\rightarrow 0}\lim_{s\rightarrow 0}\norm{F_{\delta,s}-F}_{L^2(\X,\nu)}.
\end{equation} 
Finally we stress that, for any $s>0$, we have
\begin{align}
e^{-(s/2)Q^{-1}}(\X) \subseteq Q_s^{1/2}(\X).\label{Fii}
\end{align}
Indeed by the analyticity of $e^{-(s/2)Q^{-1}}$ and by \cite[Proposition 2.1.1(i)]{LUN1}  the range of $e^{-(s/2)Q^{-1}}$ is contained in the domain of $Q^{-k}$ for every $k\in\N$. So to prove \eqref{Fii}  it is sufficient to prove that $\Id-e^{-sQ^{-1}}$ is invertible. Since $-Q^{-1}$ is negative, we have $\|e^{-sQ^{-1}}\|_{\mathcal{L}(\X)}<1$, and so $\Id-e^{-sQ^{-1}}$ is invertible. In particular $Q_s^{1/2}(\X)= Q^{1/2}(\X)$ and so we get \eqref{Fii}. Hence by the Cameron-Martin formula, for any $s,\delta>0$, we have
\[
F_{\delta,s}(x)=\int_\X F(y)e^{\scal{Q_s^{-\frac{1}{2}}e^{-sQ^{-1}}x}{y}+\frac{1}{2}\norm{Q_s^{-\frac{1}{2}}e^{-sQ^{-1}}x}^2}\mathcal{N}(0,Q_s)(dy);
\]
and so $F_{\delta,s}$ is Gateaux differentiable.

\subsection{Proof of theorem \ref{identif}}\label{core}

Finally we can prove Theorem \ref{identif}.

\begin{proof}[Proof of Theorem \ref{identif}]

Let $f\in C^1_b(\X)$. For any $\delta,s>0$, we set
\[
\varphi_{\delta,s}(x):=\int^{+\infty}_0e^{-\lambda t}P_{\delta,s}(t)f(x)dt,\quad x\in\X,
\]
where $\lambda>0$, $P_{\delta,s}(t)$ is the transition semigroup of the equation
\begin{gather*}
\eqsys{
dX(t,x)=\big(AX(t,x)+F_{\delta,s}(X(t,x))\big)dt+\sqrt{C}dW(t), & t\geq 0;\\
X(0,x)=x.
}
\end{gather*}
and $F_{\delta,s}$ is the function defined in Subsection \ref{reF}. Since $F_{\delta,s}$ is Lipschitz continuous, Gateaux differentiable  and $F_{\delta,s}-\zeta_2\Id$ is dissipative, by \cite[Step 1. of Proof of Theorem 1.4]{BF2} or \cite[Theorem 3.21]{DA4} or \cite[Proof of Theorem 11.2.13]{DA-ZA1}, for any $\delta>0$ and $s>0$, we have that $\varphi_{n,s}\in\Dom(L_{b,2})\cap C^1_b(\X)$ and 
\[
\lambda\varphi_{\delta,s}-L_{b,2}\varphi_{\delta,s}-\scal{\nabla \varphi_{\delta,s}}{F_{\delta,s}}=f, 
\]
\begin{equation}\label{uuu}
\norm{\D \varphi_{\delta,s}}_{C_b(\X)}\leq \dfrac{1}{\lambda-\zeta}\norm{\D f}_{C_b(\X)},
\end{equation}
where $\zeta$ is the constant of Hypotheses \ref{EU2}\eqref{EU2.4}. By Proposition \ref{inc}, for any $\delta>0$ and $s>0$, we have $\varphi_{\delta,s}\in\Dom(\overline{N}_0)$ and
\[
\lambda\varphi_{\delta,s}-\overline{N_0}\varphi_{\delta,s}=f+\scal{\nabla\varphi_{\delta,s}}{F_{\delta,s}-F},
\]
and by Lemma \ref{inc0}
\[
\lambda\varphi_{\delta,s}-N_2\varphi_{\delta,s}=f+\scal{\D \varphi_{\delta,s}}{F_{\delta,s}-F},
\]
where the equality holds in $L^2(\X,\nu)$. Noticing that the right hand side of \eqref{uuu} does not depend on $s,\delta$ and using \eqref{555}, we deduce that, for any $f\in C^1_b(\X)$, there exists a family $\{\varphi_{\delta,s}\,|\,n\in\N,s>0\}\subseteq\Dom(\overline{N}_0)$, such that 
\begin{equation*}
\lim_{\delta\rightarrow 0}\lim_{s\rightarrow 0}(\lambda\Id-N_2)\varphi_{\delta,s}=f,\quad \mbox{ in } L^2(X,\nu).
\end{equation*}
Since $C^1_b(\X)$ is dense in $L^2(\X,\nu)$, then $(\lambda\Id-N_2)(\Dom(\overline{N}_0))$ is dense in $L^2(\X,\nu)$, so by\cite[Proposition 3.7]{BF2} we obtain $\Dom(\overline{N}_0)=\Dom(N_2)$.
\end{proof}

\section{Dirichlet semigroup associated to a dissipative gradient systems}\label{gradsist}

Assume that Hypotheses \ref {EU3var} hold true. Let $\g\subseteq \X$ be an open set such that $\nu(\g)>0$. We consider the Dirichlet semigroup associated to the SPDE \eqref{eqFO}, defined by
\begin{equation*}
P^{\mathcal{O}}(t)\varphi(x):=\mathbb{E}\left[\varphi(X(t,x))\mathbb{I}_{\{\tau_x> t\}}\right]:=\int_\Omega\varphi(X(t,x))\mathbb{I}_{\{\tau_x> t\}}d\mathbb{P},\quad \varphi\in B_b(\mathcal{O}),
\end{equation*} 
where $X(t,x)$ is the generalized mild solution of the SPDE \eqref{eqFO} and $\tau_x$ is the stopping time defined by
\begin{equation}\label{ta}
\tau_x=\inf\{ t> 0\; : \; X(t,x)\in \g^c \}.
\end{equation} 
\begin{prop}\label{este2}
Assume that Hypotheses \ref {EU3var} hold true. For any $\varphi\in B_b(\g)$, $t>0$, we have
\[
\int_{\g}(P^{\g}(t)\varphi)^2d\nu\leq \int_{\g}\varphi^2d\nu.
\]
\end{prop}
\begin{proof}
Let
\[
\widehat{\varphi}(x)=
\begin{cases}
\varphi(x) & x\in\g,\\
0 & x\in\g^c.
\end{cases}
\]
By the H\"older inequality, we have
\[
(P^{\g}(t)\varphi)^2\leq \mathbb{E}[\varphi^2(X(t,x))\mathbb{I}_{\{\tau_x\geq t\}}]\leq \mathbb{E}[\widehat{\varphi}^2(X(t,x))\mathbb{I}_{\{\tau_x\geq t\}}]=P(t)(\widehat{\varphi}^2).
\]
Since $\nu$ is invariant for $P(t)$ and $P(t)$ is non-negative (see definition at beginning of Subsection \ref{P3}), then we conclude
\begin{align*}
\int_{\g}(P^{\g}(t)\varphi)^2d\nu\leq \int_{\g}P(t)(\widehat{\varphi}^2)d\nu\leq \int_{\X}P(t)(\widehat{\varphi}^2)d\nu=\int_\X\widehat{\varphi}^2d\nu= \int_{\g}\varphi^2d\nu.
\end{align*}
\end{proof}
\noindent By Proposition \ref{este2} the semigroup $P^{\g}(t)$ is uniquely extendable to a strongly continuous semigroup $P_2^{\g}(t)$ in $L^2(\g,\nu)$.

\begin{rmk}
By the same arguments $P^{\g}(t)$ is uniquely extendable to a strongly continuous semigroup in $L^p(\g,\nu)$, for any $p\geq 1$.
\end{rmk}
\begin{defi}
We denote by $M_2$ the infinitesimal generator of $P_2^{\g}(t)$.
\end{defi}
In this section we prove Theorem \ref{identifDiri}. We use the technique presented in \cite{AS-VA1, DA-LU1} in the case $F=0$. To do that, we have to introduce a quadratic form associated to $N_2$. Now we state the hypotheses and the preliminary results that will allow us to use such approach.

\subsection{Sobolev spaces}
\begin{hyp}\label{Sobmu}
Assume that Hypotheses \ref{EU3var} hold true and the following conditions hold true.
\begin{enumerate}[\rm(i)]
\item $\rm{Ker}(C)=\{0\}$.

\item  $A:\Dom(A)\subset\X\ra\X$ is self-adjoint and there exist $w>0$ and $M>0$ such that 
\[
\norm{e^{tA}}_{\mathcal{L}(\X)}\leq Me^{-wt}, \quad t\geq 0.
\]

\item $C(\Dom(A))\subseteq \Dom(A)$, and $ACx=CAx$ for any $x\in\Dom(A)$.

\end{enumerate}
\end{hyp} 

Under Hypotheses \ref{Sobmu} the operator
\[
Q_\infty=\int_0^\infty e^{tA}Ce^{tA^*}dt
\]
is a positive and trace class operator. 
Let $\mu$ the gaussian measure of mean $0$ and covariance operator $Q_\infty$ (we refer to \cite{BOGIE1} for a discussion of Gaussian measures in infinite dimension). Under Hypotheses \ref{Sobmu}, by \cite[Proof of Proposition 10.1.2]{DA-ZA1}, the operator
\[
C^{1/2}\D:\xi_A(\X)\subseteq L^2(\X,\mu)\ra L^2(\X,\mu,\X)
\]
is closable and we denote by $W_{C}^{1,2}(\X,\mu)$ its domain. 
\begin{hyp}\label{DIRI}
Assume that Hypotheses \ref{Sobmu} hold true and there exists  a lower semicontinuous function $U:\X\ra \R$ such that 
\begin{enumerate}

\item $\norm{x}^2e^{-2U}\in L^1(\X,\mu)$ and $e^{-2U}\in W^{1,2}(\X,\mu)$;
\item $F=-C\D U$.
\end{enumerate}
\end{hyp} 

Under Hypotheses \ref{DIRI} the SPDE \eqref{eqFO} becomes
\begin{gather*}
\eqsys{
dX(t,x)=\big(AX(t,x)-C\D U(X(t,x))\big)dt+\sqrt{C}dW(t), & t\in [0,T];\\
X(0,x)=x\in \X,
}
\end{gather*}
the operator \eqref{OPFO} reads as
\begin{equation*}
N_0\varphi(x):=\frac{1}{2}\tr[C\D^2\varphi(x)]+\scal{x}{A^*\D\varphi(x)}-\scal{C\D U(x)}{\D\varphi(x)},\qquad \varphi\in \xi_A(\X),\ x\in\X,
\end{equation*}
Moreover the following results are verified.
\begin{prop}\label{qua2}
Assume that Hypotheses \ref{DIRI} hold true.
\begin{enumerate}
\item The invariant measure $\nu$ of $P(t)$ has the form
\begin{equation*}
\nu(dx)=\frac{e^{-2U(x)}}{B}\mu(dx),\quad B:=\int_\X e^{-2U(x)}\mu(dx).
\end{equation*}
\item The operator
\[
C^{1/2}\D:\xi_A(\X)\subseteq L^2(\X,\nu)\ra L^2(\X,\nu,\X)
\]
is closable, and we denote by $W_{C}^{1,2}(\X,\nu)$ its domain.

\item For any $\varphi\in W_C^{1,2}(\X,\nu)$ and $\psi\in \Dom(N_2)$ we have
\begin{equation*}
\int_{\X} (N_2\psi)\varphi d\nu=\mathcal{Q}_2(\varphi,\psi):=-\frac{1}{2}\int_{\X}\langle C^{1/2}\D\varphi,C^{1/2}\D\psi\rangle d\nu.
\end{equation*}
\end{enumerate}
\end{prop}
We refer to  \cite[Sections 3-4]{DA-TU2} or \cite{DA-TU1} for a proof. Similarly to \cite[Section 2]{DA-LU1}, we define the following space.
\begin{rmk}
By Hypotheses \ref{DIRI} $Q_\infty$ is positive, so $\mu$ is non-degenerate Gaussian measure. Hence $\nu$ is non-degenerate namely $\nu(\g)>0$ for any $\g\in\mathcal{B}(\X)$.
\end{rmk}
\begin{defi}\label{sobsob} $ $
 We denote by $\mathring{W}_{C}^{1,2}(\mathcal{O},\nu)$ the space of the functions $u:\mathcal{O}\longrightarrow \R$ such the extension $\widehat{u}:\X\ra \R$ defined by
\[
\widehat{u}(x)=
\begin{cases}
u(x) & x\in\g\\
0 & x\in\g^c
\end{cases}
\]
belongs to $W_{C}^{1,2}(\X,\nu)$.
\end{defi}

Now we define a quadratic form on $\mathring{W}_{C}^{1,2}(\mathcal{O},\nu)$.
\begin{defi}
We denote by $\mathcal{Q}_2^{\mathcal{O}}$ the quadratic form defined by
\begin{align*}
\mathcal{Q}_2^{\mathcal{O}}(\varphi,\psi)&:=
-\frac{1}{2}\int_{\X}\scal{C^{1/2}\D\widehat{\varphi}}{C^{1/2}\D\widehat{\psi}}d\nu,\qquad \varphi,\psi\in \mathring{W}_{C}^{1,2}(\mathcal{O},\nu).
\end{align*}
Moreover, we denote by $N^\g_2$ the self-adjoint and dissipative operator associated to $\mathcal{Q}_2^{\mathcal{O}}$, namely
\begin{align*}
\Dom(N^\g_2):=\{\varphi\in\mathring{W}_{C}^{1,2}(\mathcal{O},\nu)\; &:\; \exists \beta\in L^2(\g,\nu) \mbox{ s.t. } \int_{\g}\beta\psi d\nu=\mathcal{Q}_2^{\mathcal{O}}(\beta,\psi)\;\; \forall\psi\in \mathring{W}_{C}^{1,2}(\X,\nu) \},\\
&N^\g_2\varphi=\beta,\quad \varphi\in\Dom(N^\g_2).
\end{align*}
\end{defi}

In section \ref{ingO} we shall prove Theorem \ref{identifDiri} using the following idea. For $\lambda>0$ and $f\in L^2(\g,\nu)$, we consider the equation with unknown $\varphi\in \mathring{W}^{1,2}_C(\g,\nu)$,
\begin{equation}\label{pw}
\lambda\int_\g\varphi vd\nu-Q_2^\g(\varphi,v)=\int_\g f vd\nu,\quad v\in \mathring{W}^{1,2}_C(\g,\nu). 
\end{equation}

\noindent Since the quadratic form $-\mathcal{Q}_2^{\mathcal{O}}$ is continuous, nonnegative, coercive and symmetric, by the Lax--Milgram Theorem for every $\lambda>0$ and $f\in L^2(\mathcal{O}, \nu)$ there exists a unique solution $\varphi\in \mathring{W}_{C}^{1,2}(\mathcal{O},\nu)$ of (\ref{pw}). By definition of $N^\g_2$, for every $\lambda>0$ and $f\in L^2(\g,\nu)$,  we have 
\[
R(\lambda,N_2^\g)f=\varphi,
\]
where $\varphi$ is the unique solution of (\ref{pw}). In Subsection \ref{ingO}, we will prove that
\[
R(\lambda,M_2)=R(\lambda, N^\g_2),
\]
which yields Theorem \ref{identifDiri}.
\begin{rmk}
We stress that in the trivial case, where $\g=\X$, Theorem \ref{identifDiri} follows directly from Theorem \ref{identif} and Proposition \ref{qua2}.
\end{rmk}

\subsection{The approximating semigroups}
In this section we define and study the Feynman-Kac approximating semigroup  for the semigroup $P_2^\g$. For $\epsilon>0$ we define
\begin{enumerate}
\item the set 
\begin{equation}\label{Oep}
\g_\epsilon:=\{ x\in\g\; |\; d(x,\g^c)>\epsilon\};
\end{equation}
\item the function
\begin{equation}\label{Vep}
V_\epsilon(x):=\left(\frac{1}{\epsilon}d(x,\g_\epsilon)\right)\wedge 1,\quad x\in\X.
\end{equation}
We note that $V\in C_b(\X)$, $V\equiv 0$ on $\overline{\g_\epsilon}$ and $V\equiv 1$ on $\g^c$;
\item the semigroup
\[
P^\epsilon(t)\varphi(x)=\E\left[\varphi(X(t,x))e^{-\frac{1}{\epsilon}\int_0^tV_\epsilon(X(s,x))ds}\right],\quad \varphi\in B_b(\X),\; x\in\X,
\]
where $\mii{X}$ is the mild solution of the SPDE \eqref{eqFO}.
\end{enumerate}
First we prove that $P^{\epsilon}(t)$ is uniquely extendable to a strongly continuous semigroup in $L^2(\X,\nu)$.
\begin{lemm}\label{este}
For any $\varphi\in C_b(\X)$ we have
\[
\Vert P^{\epsilon}(t)\varphi\Vert_{L^2(\X,\nu)}\leq \Vert \varphi\Vert_{L^2(\X,\nu)}.
\]
\end{lemm}
\begin{proof}
By the H\"older inequality and the fact that $V$ is nonnegative on $\X$,
\[
\vert P^{\epsilon}(t)\varphi(x)\vert^2\leq \mathbb{E}[\varphi^2(X(t,x))e^{-\frac{2}{\epsilon}\int_{\X}V_\epsilon(X(s,x))ds}]\leq P(t)(\varphi^2)(x),\quad x\in\X.
\]
Hence, since $\nu$ is invariant for $P(t)$, we have
\[
\int_{\X}\vert P^{\epsilon}(t)\varphi(x)\vert^2\nu(dx)\leq\int_{\X}P(t)(\varphi^2)(x)\nu(dx)\leq\int_{\X}\varphi^2(x)\nu(dx).
\]
\end{proof}
\noindent By Lemma \eqref{este}, the semigroup $P^{\epsilon}(t)$ is uniquely extendable in $L^2(X,\nu)$ to a strongly continuous and contraction semigroup $P_2^{\epsilon}(t)$. We denote by $N^{\epsilon}_2$ its infinitesimal generator. We recall that $N_2$ is both the closure of the operator $N_0$ in $L^2(\X,\nu)$ and the infinitesimal generator of $P_2(t)$ (see Subsection \ref{core}).
\begin{prop}\label{peps}$ $
Let $\lambda>0$ and $f\in L^2(\X,\nu)$. Then the equation
\begin{equation}\label{phie}
\lambda\varphi_{\epsilon}-N_2\varphi_{\epsilon}+\frac{1}{\epsilon}V_\epsilon\varphi_{\epsilon}=f
\end{equation}
has a unique solution $\varphi_{\epsilon}\in \Dom(N_2)$. Moreover the following estimates are verified
\begin{equation}\label{e1}
\Vert\varphi_{\epsilon}\Vert^2_{L^2(\X,\nu)}\leq\frac{1}{\lambda^2}\Vert f\Vert^2_{L^2(\X,\nu)}.
\end{equation}
\begin{equation}\label{e2}
\Vert C\D\varphi_{\epsilon}\Vert^2_{L^2(\X,\nu,\X)}\leq\frac{2}{\lambda}\Vert f\Vert_{L^2(\X,\nu)}.
\end{equation}
\begin{equation}\label{e3}
\int_{\g_\epsilon^c}V_\epsilon\varphi^2_{\epsilon}d\nu\leq\frac{\epsilon}{\lambda}\Vert f\Vert^2_{L^2(\X,\nu)}.
\end{equation}
\end{prop}

\begin{proof}
By Proposition \ref{qua2}, $N_2$ is maximal dissipative. Let $G:L^2(\X,\nu)\ra L^2(\X,\nu)$ be the operator defined by
\[
G\varphi:=\frac{1}{\epsilon}V_\epsilon\cdot\varphi,
\] 
then $-G$ is dissipative. So the operator $K:\Dom(N_2)\ra L^2(\X,\nu)$, defined by
\[
K\varphi :=N_2\varphi-G\varphi
\]
is maximal dissipative. Therefore \eqref{phie} has a unique solution $\varphi_{\epsilon}\in \Dom(N_2)$ and \eqref{e1} is verified. Multiplying both sides of (\ref{phie}) by $\varphi_{\epsilon}$, integrating over $\X$, and taking into account \eqref{qua2}, we obtain
\[
\lambda\Vert\varphi_{\epsilon}\Vert^2_{L^2(\X,\nu)}+\frac{1}{2}\Vert C^{1/2}\D\varphi_{\epsilon}\Vert^2_{L^2(\X,\nu)}+\frac{1}{\epsilon}\int_{\g_\epsilon^c}V_\epsilon\varphi^2_{\epsilon}d\nu=\int_{\X}f\varphi_{\epsilon}d\nu.
\]
By the H\"older inequality $\int_{\X}\vert f\varphi_{\epsilon}\vert d\nu\leq \Vert f\Vert_{L^2(\X,\nu)} \Vert \varphi_{\epsilon}\Vert_{L^2(\X,\nu)}$ and, by estimate (\ref{e1}), we obtain $\int_{\X}\vert f\varphi_{\epsilon}\vert d\nu\leq\frac{1}{\lambda}\Vert f\Vert^2_{L^2(\X,\nu)}$. Then 
\[
\frac{1}{2}\Vert C^{1/2}\D\varphi_{\epsilon}\Vert^2_{L^2(\X,\nu)}+\frac{1}{\epsilon}\int_{\g_\epsilon^c}V_\epsilon\varphi^2_{\epsilon}d\nu\leq \frac{1}{\lambda}\Vert f\Vert^2_{L^2(\X,\nu)}.
\]
which yields (\ref{e2}) and (\ref{e3}).
\end{proof}
\noindent Now we characterize  $N_2^{\epsilon}$. 
\begin{prop}\label{giap}$ $
For any $\epsilon>0$, we have $\Dom(N_2^{\epsilon})=\Dom(N_2)$ and
\begin{equation}\label{gie}
N_2^{\epsilon}\varphi=N_2\varphi-\frac{1}{\epsilon}V_\epsilon\varphi,\quad\forall\varphi\in \Dom(N_2).
\end{equation}
\end{prop}
\begin{proof}
First we prove that $\Dom(N_2)\subset \Dom(N_2^{\epsilon})$. We begin to show that $\xi_A(\X)\subset \Dom(N_2^{\epsilon})$. Let $\varphi\in\xi_A(\X)$. For any $x\in \X$ and $h>0$, we have
\begin{equation}\label{semig}
P_2^{\epsilon}(h)\varphi(x)-\varphi(x)=P_2(h)\varphi(x)-\varphi(x)+\mathbb{E}[(e^{-\frac{1}{\epsilon}\int_0^hV_\epsilon(X(s,x))ds}-1)\varphi(X(h,x))].
\end{equation}
Dividing both sides of (\ref{semig}) by $h>0$ we obtain
\[
\dfrac{P_2^{\epsilon}(h)\varphi(x)-\varphi(x)}{h}=\dfrac{P_2(h)\varphi(x)-\varphi(x)+\mathbb{E}[(e^{-\frac{1}{\epsilon}\int_0^hV_\epsilon(X(s,x))ds}-1)\varphi(X(h,x))]}{h}.
\]
By Theorem \ref{identif}, we know that
\begin{equation}\label{x1}
\lim_{h\rightarrow 0}\dfrac{P_2(h)\varphi-\varphi}{h}=N_2\varphi,\quad \mbox{in } L^2(\X,\nu)
\end{equation}
Since the generalized mild solution $\mi{X}$ is a $\X$-valued continuous process (see Theorem \ref{solMild} and Definition \ref{spacep}), then the functions $r\longrightarrow V_\epsilon(X(r,x))$ and $r\longrightarrow \varphi(X(r,x))$ are paths continuous, and so, recalling that $V_\epsilon\in C_b(\X)$, for any $x\in\X$ we have 
\[
\lim_{h\rightarrow 0}\dfrac{\mathbb{E}[(e^{-\frac{1}{\epsilon}\int_0^hV_\epsilon(X(s,x))ds}-1)\varphi(X(h,x))]}{h}=-\frac{1}{\epsilon}V_\epsilon(x)\varphi(x).
\]
Hence by the Dominated Convergence theorem it follows that
\begin{equation}\label{x2}
\lim_{h\rightarrow 0}\dfrac{\mathbb{E}[(e^{-\frac{1}{\epsilon}\int_0^hV_\epsilon(X(s,\cdot))ds}-1)\varphi(X(h,\cdot))]}{h}=-\frac{1}{\epsilon}V_\epsilon(\cdot)\varphi(\cdot),\quad \mbox{in } L^2(\X,\nu).
\end{equation}
So by \eqref{x1} and \eqref{x2} we obtain
\[
  N_2^{\epsilon}\varphi=\lim_{h\rightarrow 0}\frac{P^{\epsilon}(h)\varphi-\varphi}{h}=N_2\varphi-\frac{1}{\epsilon}V_\epsilon\varphi,\quad in\;\;L^2(\X,\nu).
\]
Then, for any $\varphi\in \xi_A(\X)$, we have $\varphi\in D(N_2^{\epsilon})$ and
\[
N_2^{\epsilon}\varphi=N_2\varphi-\frac{1}{\epsilon}V_\epsilon\varphi.
\]
Let now $\varphi\in \Dom(N_2)$. By Theorem \ref{identif}, $\xi_A(\X)$ is a core for $N_2$, so we can take a sequence $(\varphi_n)_{n\in\N}\subset\xi_A(\X)$ such that 
\[
\lim_{n\rightarrow+\infty}\varphi_n=\varphi,\;\;\lim_{n\rightarrow+\infty}N_2\varphi_n=N_2\varphi,\quad in\;\;L^2(\X,\nu).
\]
Since $V_\epsilon$ is bounded, we have
\[
\lim_{n\rightarrow+\infty}\frac{1}{\epsilon}V_\epsilon\varphi_n=\frac{1}{\epsilon}V_\epsilon\varphi \quad \mbox{in } L^2(\X,\nu).
\]
Hence
\begin{align*}
\lim_{n\rightarrow+\infty}N_2^{\epsilon}\varphi_n=\lim_{n\rightarrow+\infty}N_2\varphi_n-\frac{1}{\epsilon}V_\epsilon\varphi_n=N_2\varphi-\frac{1}{\epsilon}V_\epsilon\varphi,\quad \mbox{in } L^2(\X,\nu).
\end{align*}
Then, for any $\varphi\in \Dom(N_2)$, we have $\varphi\in \Dom(N_2^{\epsilon})$ and \eqref{gie} holds.

Finally we prove that $\Dom(N_2^{\epsilon})\subset \Dom(N_2)$. For any $\varphi\in \Dom(N_2^{\epsilon})$, let $\varphi_\epsilon$ be the unique solution of \eqref{phie} with $f=\lambda\varphi_\epsilon-N_2^{\epsilon}\varphi_\epsilon$. Then, by Proposition \ref{peps}, $\varphi_{\epsilon}\in \Dom(N_2)\subset \Dom(N_2^{\epsilon})$. Moreover $R(\lambda,N_2^{\epsilon})f=\varphi_{\epsilon}=\varphi$ and this concludes the proof.
\end{proof}
Finally we prove that the semigroups $P_2^\epsilon(t)$ approximate $P_2^{\g}(t)$ in $L^2(\g,\nu)$.
\begin{prop}\label{binbin}
For any $f\in L^2(\g,\nu)$ and $t>0$, we have
\begin{equation}\label{p1}
\lim_{\epsilon\rightarrow 0}(P_2^{\epsilon}(t)\widehat{f})_{|\g}=P_2^{\g}(t)f\quad in\;\; L^2(\g,\nu),
\end{equation} and, for any $\lambda>0$,
\begin{equation}\label{p2}
\lim_{\epsilon\rightarrow 0}(R(\lambda,N_2^{\epsilon})\widehat{f})_{|\g}=R(\lambda,M_2)f \quad in\;\; L^2(\g,\nu),
\end{equation}
where $\widehat{f}$ is defined in Definition \ref{sobsob}.
\end{prop}
\begin{proof}
We split the proof in two steps. As a first step we prove that for any $\varphi\in C_b(\X)$ we have
\begin{equation}\label{c1}
\lim_{\epsilon\rightarrow 0}\Vert P_2^{\epsilon}(t)\varphi - P_2^{\g}(t)(\varphi_{|\g})\Vert_{L^2(\g,\nu)}=0.
\end{equation}
And as a second step we prove the statement of Proposition.\\
\noindent\emph{Step 1.} Let $\varphi\in C_b(\X)$. First of all we prove that
\begin{equation}\label{convpun}
\lim_{\epsilon\rightarrow 0}P_2^{\epsilon}(t)\varphi(x)=P_2^{\g}(t)(\varphi_{|\g})(x)\quad x\in\g, t>0.
\end{equation}
Fixed $x\in \g$ we consider the stopping time $\tau_x$ defined in (\ref{ta}). Let $t>0$; we define the sets
\begin{align*}
&\Omega_1=\lbrace \tau_x> t\rbrace=\lbrace w\in\Omega\;|\; X(s,x)(w)\in\g,\;\forall s\in[0,t)\rbrace,\\
&\Omega_2=\lbrace \tau_x\leq t\rbrace=\lbrace w\in\Omega\;|\;\exists s_0\in(0,t]\;|\; X(s_0,x)(w)\in\g^c\rbrace.
\end{align*}
Clearly $\Omega=\Omega_1\cup\Omega_2$ and $\Omega_1,\Omega_2$ are disjoint. Fixed $x\in\g$, we have
\begin{equation}\label{disgiu}
P_2^{\epsilon}(t)\varphi(x)=\int_{\Omega_1}\varphi(X(t,x)) e^{-\frac{1}{\epsilon}\int_0^tV_\epsilon(X(s,x))ds} d\mathbb{P}+\int_{\Omega_2}\varphi(X(t,x))e^{-\frac{1}{\epsilon}\int_0^tV_\epsilon(X(s,x))ds}d\mathbb{P}.
\end{equation}
We study separately the two summands in the right hand side of \eqref{disgiu}. On $\Omega_1$, $X(s,x)\in\g$, for any $s\in [0,t)$, and then, by definition of $V_\epsilon$ (see \ref{Vep}), there exist $\epsilon_0>0$, such that
\[
V_\epsilon(X(s,x))=0, \quad \forall\epsilon<\epsilon_0,\; \forall s\in [0,t).
\]
So for the first summand of equation \eqref{disgiu}, we have
\begin{equation}\label{cip10}
\lim_{\epsilon\rightarrow 0}\int_{\Omega_1}\varphi(X(t,x)) e^{-\frac{1}{\epsilon}\int_0^tV_\epsilon(X(s,x))ds} d\mathbb{P}=\int_{\Omega_1}\varphi(X(t,x)) d\mathbb{P}.
\end{equation}
On $\Omega_2$, by the fact that the generalized mild solution $\mi{X}$ is a $\X$-valued continuous process (see Theorem \ref{solMild} and Definition \ref{spacep}), we know that, for $\mathbb{P}$-a.a. (almost all) $w\in \Omega_2$, there exists $s_0(w)\in (0,t]$ such that 
\[
X(s_0(w),x)(w)\in \partial\g,
\]
where $\partial\g$ is the boundary of $\g$. Then by definition of $V_\epsilon$, there exists $\delta(w)>0$ such that 
\[
V_\epsilon(X(s,x))(w)\geq \frac{1}{2},\quad \forall s\in [s_0(w)-\delta(w),s_0(w)].
\]
So, for the second summand of equation \eqref{disgiu}, for $\mathbb{P}$-a.a. $w\in\Omega_2$, we have 
\begin{equation}\label{cip11}
\lim_{\epsilon\rightarrow 0}e^{-\frac{1}{\epsilon}\int_0^tV_\epsilon(X(s,x))(w)ds}\leq \lim_{\epsilon\rightarrow 0}e^{-\frac{\delta(w)}{2\epsilon}}=0.
\end{equation}
Therefore by \eqref{cip11} and the Dominated Convergence theorem, we have
\begin{equation}\label{cip12}
\lim_{\epsilon\rightarrow 0}\int_{\Omega_2}\varphi(X(t,x))e^{-\frac{1}{\epsilon}\int_0^tV_\epsilon(X(s,x))ds}d\mathbb{P}=0.
\end{equation}
Hence, \eqref{cip10} and \eqref{cip12} yield \eqref{convpun}.
Moreover, for each $x\in \g$, we have\\ $\vert P_2^{\epsilon}(t)(\widehat{\varphi})(x)\vert, \vert P_2^{\g}(t)(\varphi_{|\g})(x)\vert\leq \Vert \varphi\Vert_{\infty}$. Then, by \eqref{convpun} and the Dominated Convergence theorem, \eqref{c1} is verified.

\noindent\emph{Step 2.} Let $f\in L^2(\g,\nu)$; we prove \eqref{p1} . We recall that $C_b(\X)$ is dense in $ L^2(\X,\nu)$, so there exists a sequence $(f_n)\subset C_b(\X)$ such that, for any large $n\in\N$,
\[
\Vert \widehat{f}-f_n\Vert_{L^2(\X,\nu)}\leq \frac{1}{n}.
\]
In particular
\begin{equation}\label{c2}
\Vert f-f_{n|\g}\Vert_{L^2(\g,\nu)}=\Vert \widehat{f}-f_n\Vert_{L^2(\X,\nu)}\leq \frac{1}{n}.
\end{equation}
Therefore
\begin{align*}
\Vert P_2^{\epsilon}(t)\widehat{f}-P_2^{\g}(t)f\Vert_{L^2(\g,\nu)}\leq \Vert P_2^{\g}(t)(f-f_{n|\g})\Vert_{L^2(\g,\nu)}+&\Vert P_2^{\epsilon}(t)(\widehat{f}-f_n)\Vert_{L^2(\g,\nu)}\\+\Vert P_2^{\epsilon}(t)f_n-P_2^{\g}(t)f_{n|\g}\Vert_{L^2(\g,\nu)}.
\end{align*}
By Lemma \ref{este} and Proposition \ref{este2}, we have
\[
\Vert P_2^{\epsilon}(t)\widehat{f}-P_2^{\g}(t)f\Vert_{L^2(\g,\nu)}\leq \Vert f-f_{n|\g}\Vert_{L^2(\g,\nu)}+\Vert \widehat{f}-f_n\Vert_{L^2(\g,\nu)}+\Vert P_2^{\epsilon}(t)f_n-P_2^{\g}(t)f_{n|\g}\Vert_{L^2(\g,\nu)}.
\]
Letting $\epsilon\rightarrow 0$ and $n\rightarrow +\infty$, the first and the second summand go to zero by (\ref{c2}), and the third summand goes to zero by $Step  \;1$. We recall the following identity in $L^2(\X,\nu)$
\[
R(\lambda,N_2^\epsilon)\widehat{f}=\int_0^{+\infty}e^{-\lambda t}P_2^{\epsilon}(t)\widehat{f} dt,
\]
taking the restriction to $\g$ of both sides and using \eqref{p1} we obtain (\ref{p2}).
\end{proof}
\subsection{Proof of Theorem \ref{identifDiri}}\label{ingO}
Finally we prove Theorem \ref{identifDiri}.

\begin{proof}[Proof of Theorem \ref{identifDiri}]
First we prove that $\varphi\in \mathring{W}^{1,2}_{C}(\g,\nu)$. For $\epsilon>0$, we set 
\[
\varphi_{\epsilon}=R(\lambda,N_2^{\epsilon})\widehat{f}.
\]
By Proposition \ref{giap}, $\varphi_{\epsilon}$ is the unique solution of \eqref{phie}, with $f$ replaced by $\widehat{f}$. Moreover, by Proposition \ref{peps}(\ref{e1}-\ref{e2}), the $W^{1,2}_{C}(\X,\nu)$-norm of $\varphi_{\epsilon}$ is bounded by a constant independent of $\epsilon$. Therefore there exists a sub-sequence $(\varphi_{\epsilon_k})$ weakly convergent in $W^{1,2}_{C}(\X,\nu)$ to a function $\phi$. We have to prove that $\phi=\widehat{\varphi}$, namely $\phi_{|\g}=\varphi$ and   $\phi_{|\g^c}=0$. By Proposition \ref{binbin}\eqref{p2}, we know that 
\[
\lim_{k\rightarrow+\infty}\Vert \varphi-\varphi_{\epsilon_k|\g}\Vert_{L^2(\g,\nu)}=0,
\]
so that $\phi_{|\g}=\varphi$. Since $\varphi_{\epsilon_k}$ weakly converges to $\phi$ in $W^{1,2}_{C}(\X,\nu)$, then it weakly converges to $\phi$ in $L^2(\g,\nu)$.  Recalling that $V_{\epsilon_k}\equiv 1$ in $\g^c$ (see \eqref{Vep}) and using \eqref{e3}, we obtain
\begin{align*}
\norm{\phi}^2_{L^2(\g^c,\nu}=\int_{\g^c}\phi^2 d\nu=\limsup_{k\rightarrow +\infty}\int_{\g^c}\varphi_{\epsilon_k} \phi \;d\nu &\leq\limsup_{k\rightarrow+\infty}\left(\int_{\g^c}\varphi^2_{\epsilon_k}V_{\epsilon_k}d\nu\right)^{\frac{1}{2}} \left(\int_{\g^c}\phi^2 V_{\epsilon_k}d\nu\right)^{\frac{1}{2}}\\
&\leq\lim_{k\rightarrow+\infty}\left(\frac{\epsilon_k}{\lambda}\right)^{\frac{1}{2}}\Vert f\Vert_{L^2(H,\nu)} \left(\int_{\g^c}\phi^2 V_{\epsilon_k}d\nu\right)^{\frac{1}{2}}=0,
\end{align*}
and so $\phi_{|\g^c}=0$. Therefore, $\phi=\widehat{\varphi}\in W_{C}^{1,2}(\X,\nu)$, so that $\varphi\in \mathring{W}^{1,2}_{C}(\g,\nu)$.

Finally we prove that $\varphi$ is a solution of \eqref{pw}. Fixed $v\in \mathring{W}^{1,2}_{C}(\g,\nu)$ and $k\in\N$, we multiply both members of \eqref{gie} by $\widehat{v}$ and we integrate over $\X\backslash (\g\backslash\g_{\epsilon_k})$. Since $ V_{\epsilon_k}\widehat{v}\equiv 0$ on $\X\backslash (\g\backslash\g_{\epsilon_k})$, we have
\[
\lambda\int_{\X\backslash (\g\backslash\g_{\epsilon_{k}})}\varphi_{\epsilon_k} \widehat{v} \;d\nu+\frac{1}{2}\int_{\X\backslash (\g\backslash\g_{\epsilon_{k}})}\scal{C^{1/2}\D\varphi_{\epsilon_k}}{C^{1/2}\D\widehat{v}}\;d\nu=\int_{\X\backslash (\g\backslash\g_{\epsilon_{k}})}\widehat{f}\widehat{v}\;d\nu.
\]
Recalling the definition of $\g_{\epsilon_k}$ (see \eqref{Oep}), letting $k\rightarrow+\infty$, we obtain
\[
\lambda\int_{\X}\widehat{\varphi} \widehat{v}\;d\nu+\frac{1}{2}\int_{\X}\scal{C^{1/2}\D\widehat{\varphi}}{C^{1/2}\D\widehat{v}}\;d\nu=\int_{\mathcal{\X}}\widehat{f}\widehat{v}\;d\nu,
\]
and so we conclude that $\varphi$ satisfies \eqref{pw}. We recall that, by the Lax-Milgram theorem, the weak solution of \eqref{pw} is unique and so, for any $\lambda>0$ and $f\in L^2(\g,\nu)$, we have
\[
R(\lambda, M_2)f=R(\lambda, N_2^\g)f,
\]
and so the Theorem \ref{identifDiri} is proved.
\end{proof}

\section{Examples}
In this Section we present some examples of $A$, $C$ and $F$ that verify the hypotheses of Theorems \ref{identif} and \ref{identifDiri}.
\subsection{An example for Section \ref{gradsist}}\label{gra}
For this example we could consider a very general framework choosing as $X$, $E$, $A$ and $C$ as those defined in \cite[Chapter 6]{CER1}. However, in order not to overburden the calculations, we will consider a less general setting.

Let $\X=L^2([0,1],\lambda)$ where $\lambda$ is the Lebesgue measure and let $E=C([0,1])$. Let $A$ be the realization in $L^2([0,1])$ of the second order derivative with Dirichlet boundary condition and $C=\Id_\X$. By \cite[Section 6.1]{CER1} Hypotheses \ref{EU2}\eqref{EU2.3}, \ref{EU2}\eqref{EU2.5} and $A$ is dissipative in $C([0,1])$. Moreover the constant $w$ of Hypotheses \ref{Sobmu} is equal to $-\pi^2$ (see \cite[Chapter 4]{DA1}). By \cite[Lemma 8.2.1]{CER1} condition \eqref{condinv} of Hypotheses \ref{EU3} is verified.

Now we define the function $F$. Let $\varphi\in C^2(\R)$ be a function such that $\varphi'$ is increasing, and there exist $d_1,d_2>0$ and an  $m\in\N$ such that
\begin{align}
&\abs{\varphi'(y)}\leq d_1(1+\abs{y}^{m}),\quad y\in\R;\label{cre1}\\
&\abs{\varphi''(y)}\leq d_2(1+\abs{y}^{m-1}),\quad y\in\R;\label{cre2}.
\end{align}
Let $\zeta_2>0$. We consider the function $\phi:\R\ra\R$ defined by
\[
\phi(y)=\varphi(y)+\frac{\zeta_2}{2}y^2,
\]
and the function $U:\X\ra \R$ defined by
\[
U(f)=\begin{cases}
\int^1_0\phi(f(x)), & f\in E,\\
0, & f\not\in E.
\end{cases}
\]
In this case the operator of Hypotheses \ref{Sobmu} is $Q_\infty=A^{-1}$. Let $\mu\sim N(0,Q_\infty)$. By \cite[Proposition 5.2]{DA-LU2}, $U\in W_C^{1,p}(\X,\mu)$, for any $p\geq 1$, and 
\[
\D U(f)(x)=\phi'\circ f(x)=\varphi(f(x))+\zeta_2f(x),\quad \forall f\in E=C([0,1]),\; x\in [0,1].
\]
We set $F=-\D U$, and we recall that we have taken $C=\Id_\X$. Hence  Hypotheses \ref{EU2}\eqref{EU2.1} are verified. By \eqref{cre1} and \eqref{cre2} Hypotheses \ref{EU2}\eqref{EU2.7} and \ref{EU2}\eqref{EU2.6} are verified. By \cite[Example D.7]{DA-ZA4} and standard calculations Hypotheses \ref{EU2}\eqref{EU2.4} are verified. We stress that all the hypotheses of Theorem \ref{invX} are verified, so $\nu(C([0,1]))=1$, where $\nu$ is the invariant measure of the transition semigroup associated to the generalize mild solution of \eqref{eqFO}. Finally, by the definition of $\phi$ and the Fernique theorem, the hypotheses of Theorem \ref{identifDiri} are verified.
It is also possible to consider an operator $A$ that verifies Hypotheses \ref{EU2}\eqref{EU2.4}(a) with $\zeta_1<0$(see \cite[Example 11.36]{DA-ZA4}), in this way we can take $\zeta_2<0$.
\subsection{An example where $F$ is not a Nemytskii type operator}\label{noN}

In this subsection we consider a class of functions $F$ already presented in \cite[Section 5.2]{BF2}. We recall the notion of infinite dimensional polynomial (see \cite{CHAE1,DIN1,MUJ1}).

For every $n\in\N$, we say that a map $V:\X^n\ra\X$ is $n$-multilinear if it is linear in each variable separately. A $n$-multilinear map $B$ is said to be symmetric if 
\begin{equation}\label{simme}
V(x_1,\ldots,x_n)=V(x_{\sigma(1)},\ldots,x_{\sigma(n)}),
\end{equation}
for any permutation $\sigma$ of the set $\{1,\ldots,n\}$. We say that a function $P_n:\X\ra\X$ is a homogeneous polynomial of degree $n\in\N$ if there exists a $n$-multilinear symmetric map $B$ such that for every $x\in\X$
\begin{align}\label{Komi}
P(x)=V(x,\ldots,x).
\end{align}
We consider the function $F:\X\ra\X$ defined by
\begin{align*}
F(x):=P_n(x)+\zeta_2 x,
\end{align*}
where $x\in\X$, $\zeta_2\in\R$ and $P_n$ is a homogeneous polynomial of degree $n$ such that,
\begin{align}\label{abyss}
\gen{V(h,x,\ldots,x),h}\leq 0,
\end{align}
where $V$ is the $n$-multilinear map defined by \eqref{Komi}. By \cite[Theorem 3.4]{CHAE1}, there exists $d>0$ such that
\begin{equation}\label{crepol}
\norm{F(x)}\leq d(1+\norm{x}^n),\quad x\in\X.
\end{equation}
Moreover, for any $x,h\in\X$, we have
\begin{equation*}
\J P_n(x)h=nV(h,x,\ldots,x),
\end{equation*} 
and so, by \eqref{abyss}, for any $x,y\in\X$, we obtain
\begin{equation}\label{dispol}
\scal{F(x)-F(y)}{x-y}\leq \zeta_2\norm{x-y}^2.
\end{equation}
Let now consider a particular case. Let $E=\X=L^2([0,1],\lambda)$, let $A$ be the realization in $\X$ of the second order elliptic operator defined in \cite[Section 6.1]{CER1} and let $C\in\mathcal{L}(\X)$ be the positive operator defined in \cite[Section 6.1]{CER1}. Let $K\in L^2([0,1]^4)$ and assume that $K$ is symmetric (\eqref{simme}). Let
\begin{align}\label{polpol}
[P_3(f)](\xi):=\int_0^1\int_0^1\int_0^1 K(\xi_1,\xi_2,\xi_3,\xi) f(\xi_1)f(\xi_2)f(\xi_3)d\xi_1 d\xi_2 d\xi_3
\end{align}
for $f\in L^2([0,1])$. $P$ is a homogeneous polynomial of degree three on $L^2([0,1])$ (see \cite[Exercise 1.73]{DIN1}). \eqref{abyss} holds whenever $K$ has negative value (see \cite[Section 5.2]{BF2}). So, by the same arguments used in the previous example, the hypotheses of Theorem \ref{identif} are verified. It is also possible to consider a general infinite dimensional polynomial of odd degree $n\in\N$. We remark that other choices of $A$ and $C$ are possible, for example we could consider the ones chosen in \cite{DA-LU1}.

\subsection{An application of Theorem \ref{invX}}\label{belapp}
Now we present a particular case of the previous example where $E=W^{1,2}([0,1],\lambda)$. We assume that $A=-\frac{1}{2}\Id_\X$, so it verifies Hypotheses \ref{EU2}\eqref{EU2.3} and $A+\frac{1}{2}\Id$ is dissipative in both $L^2([0,1],\lambda)$ and $W^{1,2}([0,1],\lambda)$. Let $B$ be the realization of the second order derivative in $\X$ with Dirichlet boundary conditions. We recall that $B$ is a negative operator, $\Dom((-B)^{\frac{1}{2}})=W_0^{1,2}([0,1],\lambda)$ and $(-B)^{-\gamma}$ is a trace class operator, for any $\gamma>\frac{1}{2}$ (see \cite[Section 4.1]{DA1}). Let $\beta>2$ and set $C=(-B)^{-\beta}$. Then
\begin{align*}
\norm{W_A(t)}_{W^{1,2}([0,1],\lambda)}^2&=\norm{(-B)^{1/2}\int^t_0e^{-s}B^{-\beta/2}dW(s)}_{L^2([0,1],\lambda)}\\
&=\norm{\int^t_0e^{-s}(-B)^{(1-\beta)/2}dW(s)}_{L^2([0,1],\lambda)},
\end{align*}
and so by \cite[Theorems 4.36 and 5.11]{DA-ZA4}, Hypotheses \ref{EU2}\eqref{EU2.5} and condition \eqref{condinv} of Hypotheses \ref{EU3} are verified.
 
Let $F$ be as in Subsection \ref{noN}. In addition we assume that $K$ has weak derivative with respect to the fourth variable, such that
\[
\dfrac{\partial K }{\partial \xi}\in L^2([0,1]^4,\lambda).
\]
Let $f\in W^{1,2}([0,1],\lambda)$. We have $F(f)=P_3(f)+\zeta_2f\in W^{1,2}([0,1],\lambda)$(see \eqref{polpol}) and its weak derivative is
\begin{equation}\label{dedeb}
(F(f))'=\int_0^1\int_0^1\int_0^1 \dfrac{\partial K }{\partial \xi}(\xi_1,\xi_2,\xi_3,\xi) f(\xi_1)f(\xi_2)f(\xi_3)d\xi_1 d\xi_2 d\xi_3+\zeta f'.
\end{equation}
If we assume that $(\partial K/\partial \xi)\in L^2([0,1]^4,\lambda)$ is symmetric (see \eqref{simme}) and that it has negative value, then by \eqref{dedeb} and the same arguments used in Subsection \ref{noN}, \eqref{crepol} and \eqref{dispol} are verified in $W^{1,2}([0,1],\lambda)$. Hence, by the same arguments of the previous examples, the hypotheses of Theorem \ref{invX} are verified and so $\nu(W^{1,2}([0,1],\lambda))=1$.

\vspace{0.6cm}

\noindent {\bf Acknowledgements.} The author would like to thank A. Lunardi and S. Ferrari for many useful discussions and comments.

\end{document}